\documentclass[a4paper,10pt]{siamltex} 
\usepackage{amsmath,amssymb,amsfonts,graphicx,epsfig,color,url}
\usepackage{latexsym,mathrsfs,cite} 

\usepackage{float}

\usepackage{hyperref} 
\usepackage{comment}
\usepackage[title,titletoc,toc]{appendix}

\definecolor{colorJRblue}{rgb}{0.,0.,1.}%67}
\definecolor{colorJRred}{rgb}{1.,0.,0.}%67}

\def\R{\mathbb{R}}
\def\C{\mathbb{C}}
\def\Z{\mathbb{Z}}
\def\NN{\mathbb{N}}
\def\N{\mathbb{N}_*}

\def\Lspace{{\sf L}}
\def\Hspace{{\sf H}}
\def\L{{\bf L}}
\def\d{d}
\def\crit{\mathrm{crit}}

\def\rmi{\mathrm{i}}

\def\ds{D_k}
\def\calO{\mathcal{O}}
\def\DT{\Delta T}

\newcommand{\uP}{u_1}
\newcommand{\uM}{u_2}

\newcommand{\vP}{v_1}
\newcommand{\vM}{v_2}

\newtheorem{hypothesis}{Hypothesis}

\newtheorem{remark}{Remark}

\begin{document}
\title{Global stability and local bifurcations in a two-fluid model for tokamak plasma}
\author{
        D.\ Zhelyazov \thanks{Centrum Wiskunde \& Informatica, Science Park 123, 1098 XG Amsterdam. FOM Institute DIFFER - Dutch institute for fundamental energy research, Association EURATOM-FOM, P.O. Box 1207, 3430 BE Nieuwegein, the Netherlands} \and D.\ Han-Kwan \thanks{\'Ecole Normale Sup\'erieure, D\'epartement de Math\'ematiques et Applications, 45 rue d'Ulm, 75005 Paris, France} \and J.D.M.\ Rademacher \thanks{Centrum Wiskunde \& Informatica, Science Park 123, 1098 XG Amsterdam, the Netherlands} 
}

\maketitle

\begin{abstract}
We study a two-fluid description of high and low temperature components of the electron velocity distribution of an idealized  
tokamak plasma. We refine previous results on the laminar steady-state  solution. On the one hand, we prove global stability outside a  
parameter set of possible linear instability. On the other hand, for a large set of parameters, we prove the primary instabilities for  
varying temperature difference stem from the lowest spatial harmonics. We moreover show that any codimension-one bifurcation is a  
supercritical Andronov-Hopf bifurcation, which yields stable periodic solutions in the form of traveling waves. In the degenerate case, where the instability region in the temperature difference is a point, we prove that the bifurcating periodic orbits form an arc of stable periodic solutions.  We provide numerical simulations to illustrate and corroborate our analysis. These also suggest that the stable periodic orbit, which bifurcated
from the steady-state, undergoes additional bifurcations.
%[DRAFT: \today]
\end{abstract}

%%%%%%%%%%%%%%%%%%%%%%%%
\section{Introduction}

In this paper we analyze the stability and primary local bifurcations of a  laminar steady state in the following two-fluid model for high and low temperature in a tokamak fusion plasma near the scrape-off layer. The model equations read
\begin{equation}\label{eq1}
\left\{
\begin{aligned}
\partial_{t} \rho^{+} &= T^{+} \partial_{x_2}\rho^{+} - E^{\perp} \cdot \nabla \rho^{+} + \nu \nabla^2 \rho^{+}, \\
\partial_{t} \rho^{-} &= T^{-} \partial_{x_2}\rho^{-} - E^{\perp} \cdot \nabla \rho^{-} + \nu \nabla^2 \rho^{-},\\
E &= -\nabla V,\\
-\nabla^2 V &= \rho^{+} + \rho^{-} - 1,
\end{aligned}
\right.
\end{equation}
where $E^{\perp} = \big( E_2, -E_1\big)^T$, $\nu>0$, and are posed in the cylindrical domain 
\[
x = (x_1, x_2) \in [0, L_{1}] \times \R/L_2 \Z,%\mbox{ , } t\geq 0
\]
subject to the Dirichlet boundary conditions 
\begin{equation}
\begin{aligned}
V(0,x_2,t) &= V(L_1,x_2,t) = 0,\\
\rho^{\pm}(0,x_2,t) &= \rho_{ss}^{\pm}(0)  \mbox{ , } \rho^{\pm}(L_1, x_2,t) = \rho_{ss}^{\pm}(L_1),
\end{aligned}
\end{equation}
where $\rho_{ss}^{+}(x_1) := 1 - \frac{x_1}{L_1}$, $\rho_{ss}^{-}(x_1) := \frac{x_1}{L_1}$.

For $\nu=0$ and without the Dirichlet boundary conditions on $\rho^\pm$, this system has been derived in \cite{daniel1} ($L_1=L_2$ was chosen there), to which we refer for details on the model origins. Briefly, $\rho^\pm$ model miscible phases of `hot' and `cold' plasma with constant temperatures $T^{+}  > T^{-} >0$, and $V$ the electric potential, driving $\rho^\pm$ via the drift velocity $E \times B/|B|^2$ of all charged particles. The addition of viscous terms on the one hand allow to model additional physics by adding diffusion or dissipation; on the other hand, it changes the system from hyperbolic to parabolic, whose bifurcations are easier to analyze. It turns out that $\nu>0$ allows for richer destabilization scenarios. In order to relate our results with the hyperbolic system, we include an analysis of the case of small $\nu > 0$. For the benefit of a significant simplification of the analysis, we restrict in this paper to the case of equal viscosity for $\rho^\pm$.

The introduction of viscosity requires additional boundary conditions. The Dirichlet boundary conditions on $\rho^\pm$ are suitable in this context and helpful for our analysis, though in other physical contexts these may not be the right choice. Notably, the boundary conditions allow for the laminar steady state 
\begin{eqnarray}\label{e:rhoss}
\rho_{ss} = ( \rho_{ss}^+ , \rho_{ss}^-),
\end{eqnarray}
for which the electric potential and field vanish, and whose relevance for the system was noted in \cite{daniel1} for $\nu=0$. If $\nu>0$, it is in fact the only steady state that is independent of $x_2$. In this paper, we present a detailed analysis of its stability and bifurcations for $\nu>0$. 
For moderate viscosity, the equilibrium is unstable in a bounded interval $[\Delta T_1, \Delta T_2]$ of the parameter $\Delta T = T^+-T^-$, see Figure~\ref{fig:fig1}. On the other hand, the equilibrium is stable for large enough temperature difference, $\Delta T>\Delta T_2$ and also for small enough (including negative) temperature difference $\Delta T<\Delta T_1$. The parameter $\DT$ is relevant in our analysis since it arises in the comoving variable $x_2\to x_2 - T^- t$, while $T^-$ is removed.

For the hyperbolic case $\nu=0$, it turns out that $\Delta T_1=0$ and, for the spatially lowest harmonic eigenfunction,
\[
\Delta T_2= \frac{4L_1L_2^2}{\pi^2(L_2^2+4L_1^2)} = \frac{4 \ell^2 L_1}{(4 + \ell^2) \pi^2},
\] 
where $\ell=L_2/L_1$ is the aspect ratio. At $L_1=L_2$, that is, $\ell=1$, this is the instability region already found in \cite{daniel1}. It turns out that $\nu>0$ and $L_1\neq L_2$ allows for much richer bifurcation scenarios, and it moreover explains the location of the global stability threshold $\DT_*=4L_1/\pi^2$ as the limiting linear stability threshold for $\nu=0$ via
\[
\lim_{\ell\to\infty} \Delta T_2= \DT_*.
\] 
In fact, this is the upper bound of $\DT$ for any linear instability. 

One of the original motivations for this study from \cite{daniel1} with $L_1=L_2$ was to find \emph{subcritical} bifurcations at $\Delta T_1, \Delta T_2$, which would also explain a difference between the local instability threshold $\Delta T_2$ (given by spectral stability) and the global stability threshold $\Delta T_*$ (essentially depending on a Poincar\'e inequality constant). However, it turns out that the bifurcations are always supercritical.

\medskip
Coming back to the model origins, the sign of $\Delta T$ can be related to the region within the tokamak that is modelled by \eqref{eq1}: `good curvature' (negative $\Delta T$) and `bad curvature' regions, which is consistent with the different  stability properties for positive and negative $\Delta T$ as noted in \cite{daniel1}. The model captures the Electron Temperature Gradient instability. The modelling and physical relations to L-H transition (see \cite{Wagner}) remain to be understood. ``Clearly, the model selection criteria, apart from the sound physics behind them, should be based on their capability to reproduce key experimental facts such as spontaneous L-H transitions, characteristic intermediate regimes (such as dithering), or hysteresis"\cite{Malkov}.

\medskip
In this paper, we pursue a mathematical analysis that may serve as a basis to investigate further the relations to physical phenomena. The main results may be summarized as follows, see also Figures~\ref{fig:fig1} and \ref{fig:fig2}.

\paragraph{Global stability (Theorem \ref{theo:GS})} The steady state $\rho_{ss}$ is globally $\Lspace^2$-stable for $\Delta T<0$ and $\Delta T>\Delta T_*$. Global stability for $\nu=0$ in a similar region was already proven in \cite{daniel1} via an explicit Lyapunov functional given by energy conservation. In the case of dissipation, improved bounds give the present result with exponential decay. Moreover, the global stability threshold is sharp in the sense that it is realized as a limiting linear instability threshold in parameter space.

\paragraph{Local bifurcations (Theorems \ref{theorem2}, \ref{theorem3})} For a large class of parameter configurations, including $L_2/L_1< 2\sqrt{2}\approx 2.8$, the following holds. At the stability thresholds $\Delta T_j$, $j=1,2$, the critical modes are spatially the lowest harmonics, and the system undergoes supercritical Andronov-Hopf bifurcations corresponding to periodic travelling wave bifurcations with velocity $\omega = \pi(T^+ + T^-)/L_2$. Near the bifurcations, the reduced dynamics on a center manifold is the generic normal form. For $0<\nu\ll 1$ this always holds at $\DT_2$, but not necessarily at $\DT_1$.

The  local unfolding of the degenerate case $\Delta T_1= \Delta T_2$ proves that the two branches of periodic orbits are locally connected, and form an arc of stable periodic solutions. We numerically corroborate that, further away from this degeneracy, secondary instabilities occur along the arc. See Figure~\ref{fig:fig_traces} and \S\ref{s:center}.

In case $L_2 \gg L_1$, the primary instabilities can also be higher spatial harmonics, even simultaneously. We thus suspect rich dynamics already at onset, but a detailed analysis is beyond the scope of this paper. It is also possible, that as $\Delta T$ increases, a sequence of destabilization and restabilization occur through different harmonics. Roughly speaking, a heuristic interpretation for the model background would be that increasing $L_2$ for fixed $L_1$, $\Delta T$ introduces richer bifurcations from the steady state.

\begin{figure}[H]
	\label{fig:fig1}
\begin{center}
\begin{tabular}{cc}
	\includegraphics[scale = 0.45]{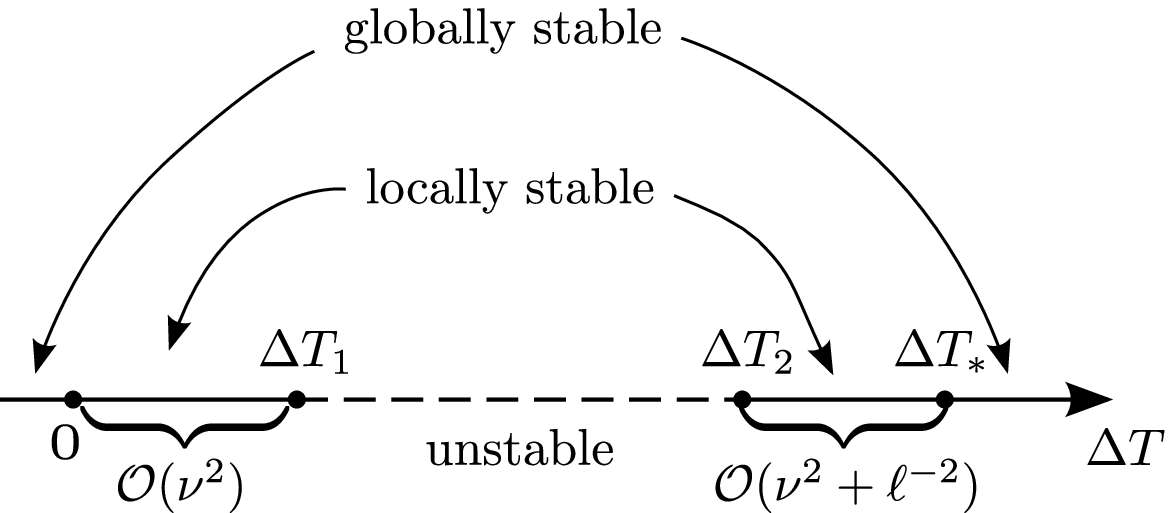}\hspace{5mm}
&	\includegraphics[scale = 0.45]{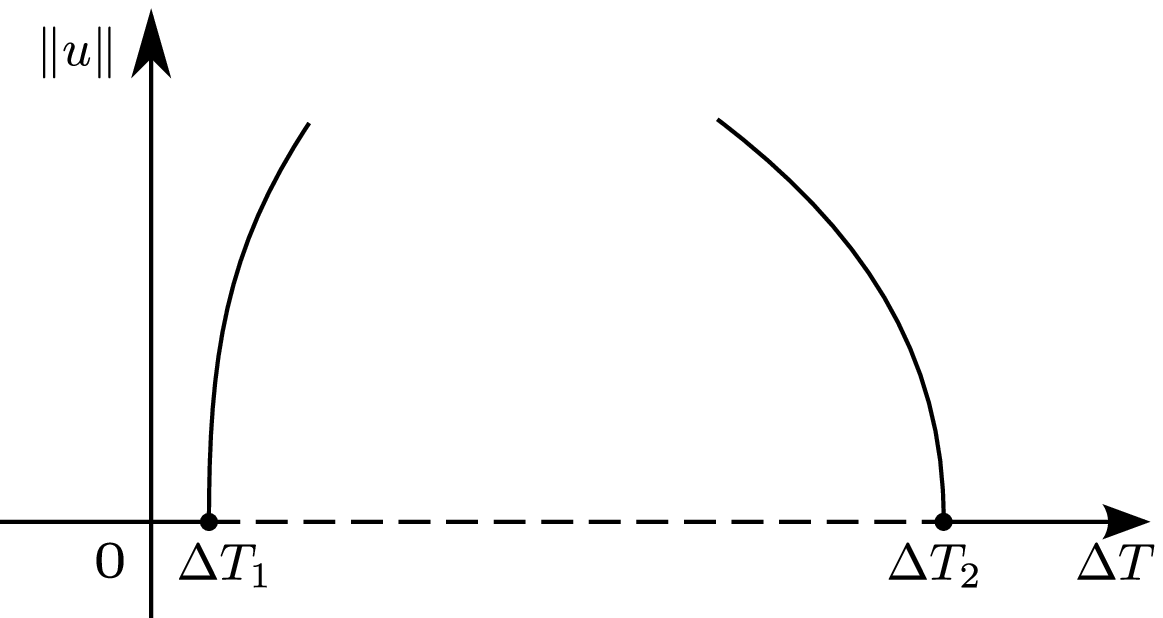}\\
(a) & (b)
\end{tabular}
	\caption{(a) Schematic illustration of the main case of a primary 1-instability region in the stability analysis of the steady-state $\rho_{ss}$ when including viscosity. The global stability threshold $\Delta T_*$ is larger than the linear stability threshold, even at $\nu=0$. However, in the limit $\nu\to 0$ the lower thresholds coincide, and if in addition $\ell\to\infty$, then also the upper linear thresholds tend to the global ones. (b) Sketch of local bifurcation diagram of the steady state $u = 0$ with supercritical branches of stable limit cycle. Solid line represents stable solutions and dashed lines unstable ones.}
	\end{center}
\end{figure}

\medskip
This paper is organized as follows. Section \ref{s:num} contains numerical computations, illustrating the results. In \S\ref{refor} we reformulate the problem for a subsequent bifurcation analysis. Section \ref{s:spec} concerns the spectrum of the linearized operator around the steady state $\rho_{ss}$. In \S\ref{s:center} we discuss the center manifold reduction, reduced vector fields and prove the main bifurcation results. In \S\ref{s:tw} we explain the relation to travelling wave bifurcations, and briefly consider pattern formation in case of an infinite strip. In \S\ref{s:nl} we discuss nonlinear instability for $\nu\geq0$ in the linearly unstable region. Finally, \S\ref{s:global} contains the global stability result.

\medskip
\textbf{Acknowledgement.} This work, supported by the European Communities under the contract of Association between EURATOM/FOM, was carried out within the framework of the European Fusion Programme with financial support from NWO. The views and opinions expressed herein do not necessarily reflect those of the European Commission. J.R.\  has been supported in part by NWO cluster NDNS+, D.H.-K.\ is grateful to the CWI, where this work was initiated, for its hospitality. We thank Hugo de Blank for his comments and suggestions on a draft version.

%%%%%%%%%%%%%%%%%%%%%
\section{Numerical results}\label{s:num}

For illustration of the upcoming analytical bifurcation results, we present in this section some numerical computations.
We compute the deviation $u=\rho-\rho_{ss}$ (see \eqref{eq2}) and discretize with a finite-dimensional spectral decomposition (see \eqref{e:gk} for the definition of the harmonics $g_k$):
\begin{equation}
u_{l}(x,t) = \sum_{k_1=1}^{N_{x_1}} \sum_{k_2=-N_{x_2}}^{N_{x_2}} C_{k_1,k_2,l}(t)g_{k_1,k_2}(x)\text{ , }l=1,2.
\end{equation}
We integrate the resulting system of ODEs using a semi-implicit Crank-Nicolson scheme, where only the linear part is implicit\footnote{We modified a code by Jean-Christophe Nave - MIT Department of Mathematics, jcnave$@$mit.edu for Navier-Stokes equations in vorticity formulation.}. We used the parameter values 
\begin{equation}\label{e:pars}
\nu = 9.10^{-4}, \; L_1 = L_2 = 2, \; T^{-} = 10^{-1},
\end{equation}
while $\Delta T=T^+-T^-$ varied across the instability region. Note that this lies in the region $L_2<2\sqrt{2}L_1$, thus the primary instabilities come from the lowest spatial harmonics as proven in Theorem~\ref{theorem1}. All the simulations were made with $N_{x_1} = N_{x_2} = 32$, though we selectively checked with $N_{x_1} = N_{x_2} = 64$.

\begin{figure}[H]
\begin{center}
\begin{tabular}{cc}
	\includegraphics[scale = 0.32]{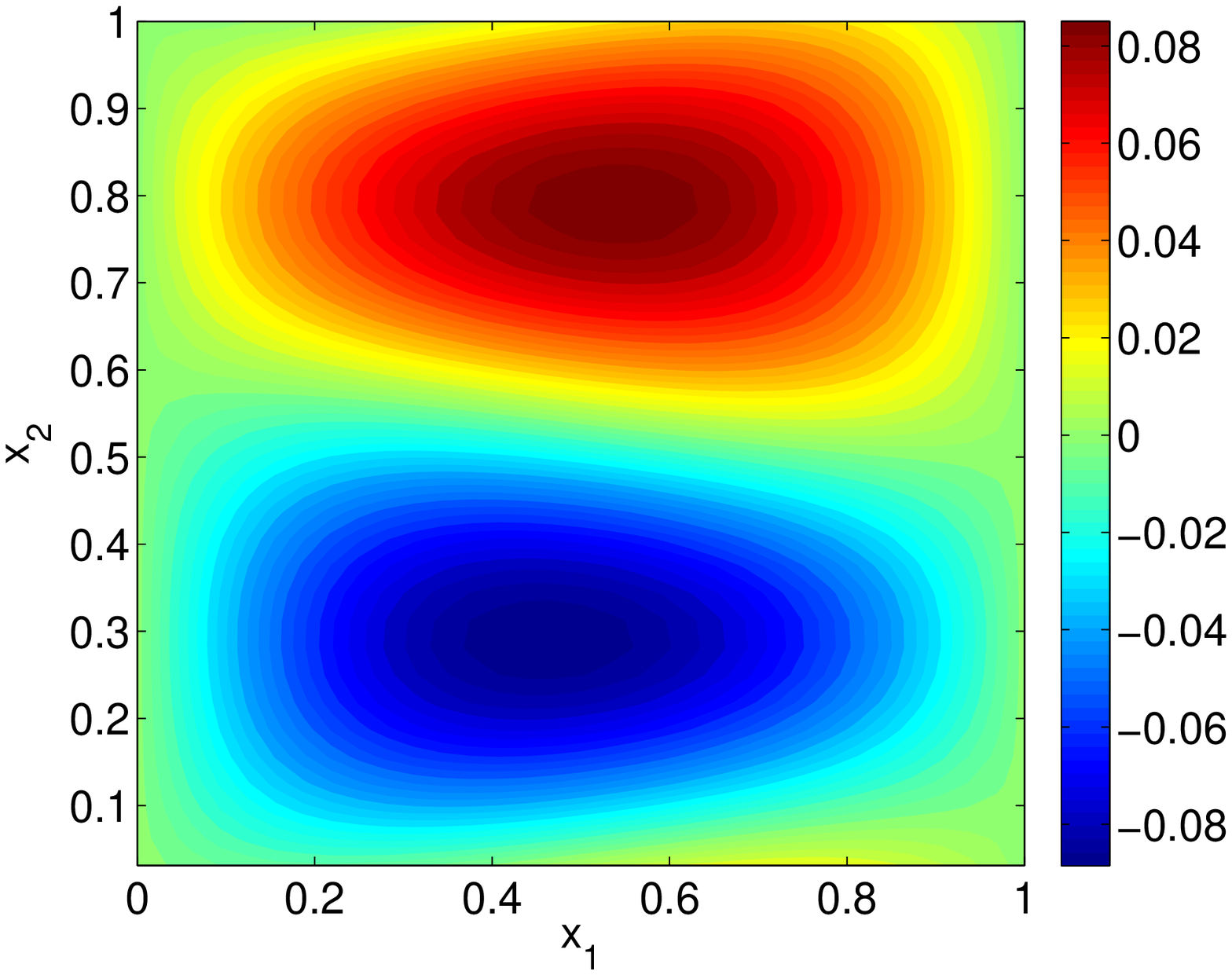}
&	\includegraphics[scale = 0.32]{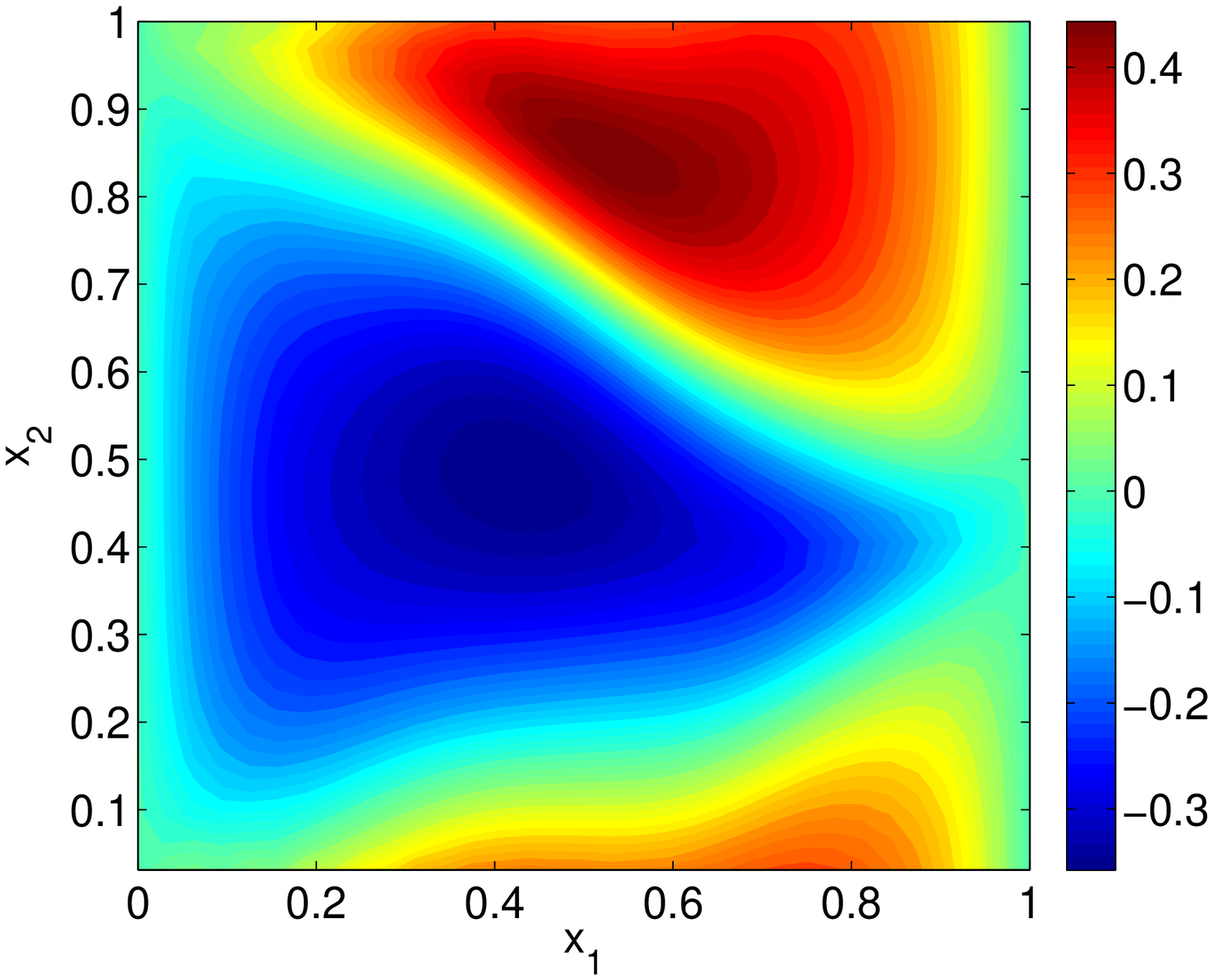}\\
(a) & (b)
\end{tabular}
\caption{Contour plots of $u_1(t_2) = \rho_1(t_2)-\rho_{ss}$ with $t_2$ sufficiently large for $\Delta T\in (\Delta T_1, \Delta T_2)$ and parameters as in \eqref{e:pars}. (a) the dynamics is a translation in the periodic $x_2$-direction, $\Delta T = 0.159291$, and (b) $\Delta T = 0.146122$, the dynamics resembles a modulated travelling wave.}
\label{fig:fig_plots}
\end{center}
\end{figure}

In Figure \ref{fig:fig_plots} we plot two periodic travelling wave solutions near the upper stability threshold $\Delta T_2 \approx 0.162$ and further inside the nonlinear regime as can be seen by the locus of parameters in Figure~\ref{fig:fig_nonlinear_attractor}(a). The weakly nonlinear solution for $\DT\approx \DT_2$ closely resembles the unstable eigenfunction, while the solution further inside the nonlinear regime has a clear nonlinear structure.

In order to trace the stable branches of solutions bifurcating from the supercritical Andronov-Hopf bifurcations at $\Delta T = \Delta T_1, \Delta T_2$, we perform a simple continuation: for $\DT$ near the bifurcation at $\DT_1$, we simulate an initial condition close to $\rho_{ss}$ and after a long transient compute the sup-norm over a long time interval. We then slightly increase $\Delta T$ and repeat this step with the initial condition being the solution at the final time of the previous step. In this way we obtain the bifurcation diagram in Figure~\ref{fig:fig_nonlinear_attractor}, where the numerical instability thresholds are in very good agreement with the analytical ones. 

\begin{figure}[H]
\centering
\begin{tabular}{cc}
	\includegraphics[scale = 0.32]{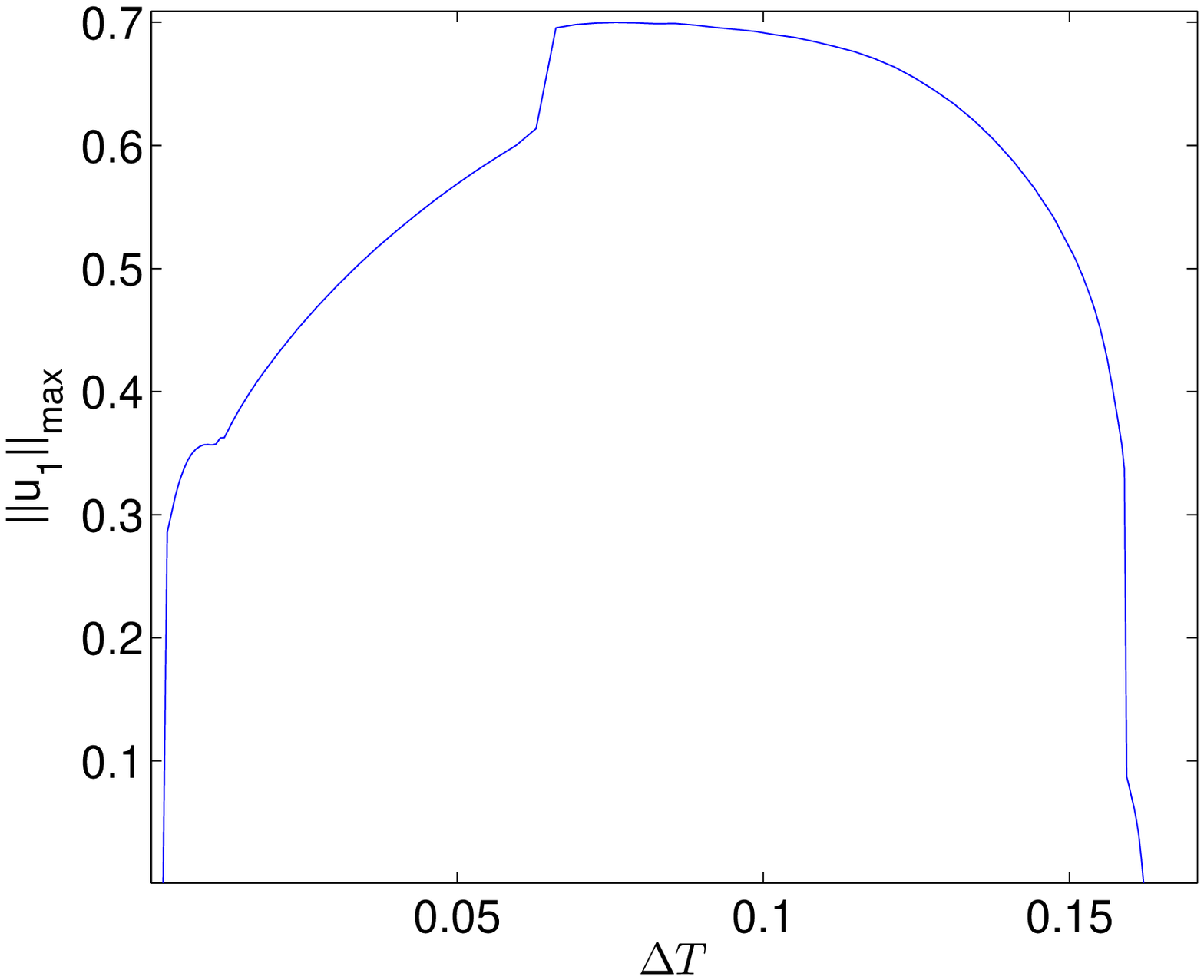}
&    \includegraphics[scale = 0.32]{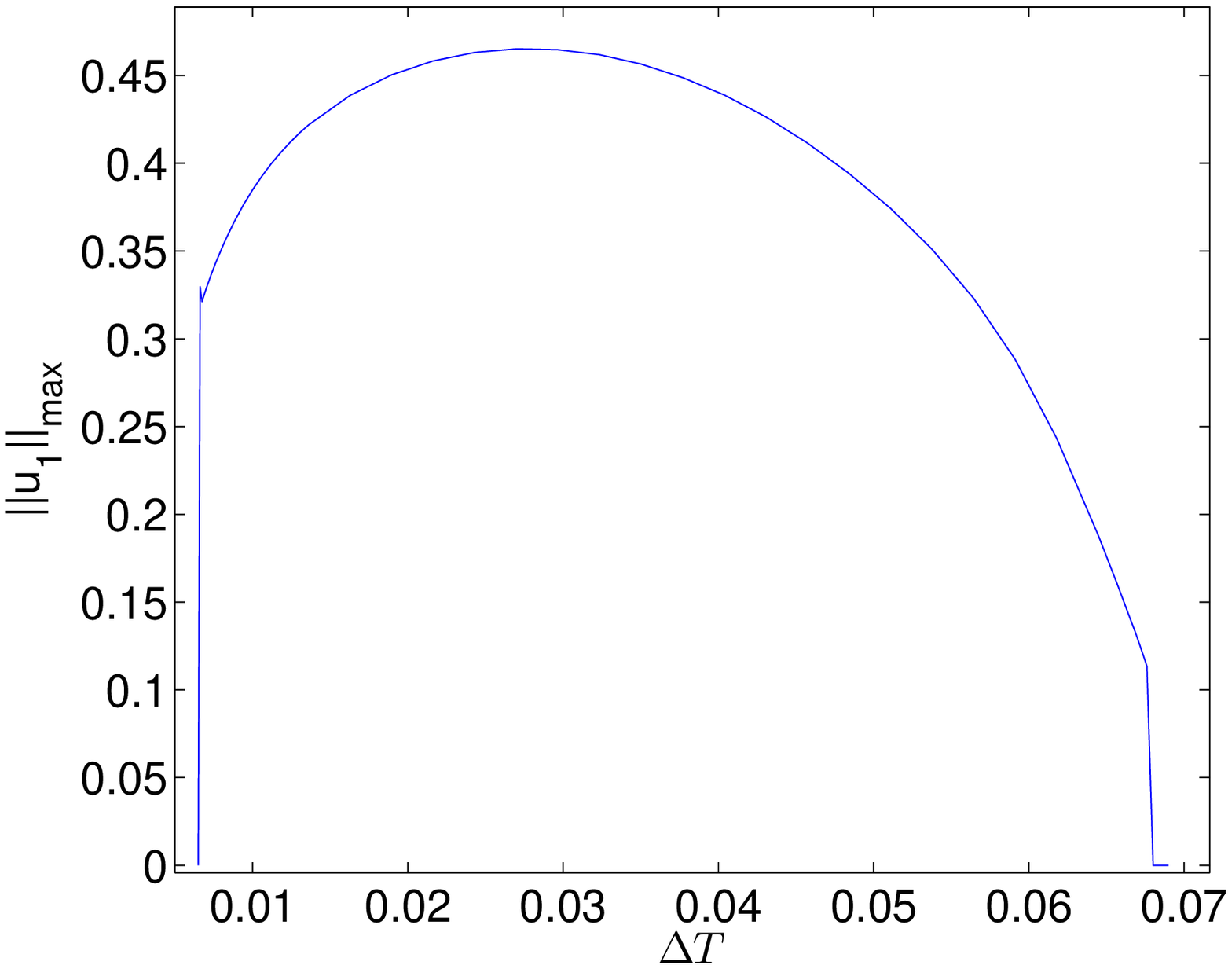}\\
(a) & (b)
\end{tabular}
	\caption {Bifurcation diagrams in $\Delta T$ with $\max_{t_1 \leq t \leq t_2} \Vert u_1(t) \Vert_{\infty}$ on the vertical axis, where $t_1,t_2$ are taken sufficiently large ($t_1 \geq 400$), so that we get a good approximation of the attractor. (a) parameters as in\ref{e:pars}, and (b) $L_1 = 0.9$, $L_2=0.9$, $\nu = 9.10^-3$, the diagram is in agreement with corollary \ref{c:hopf}.}
	\label{fig:fig_nonlinear_attractor}
\end{figure}

As predicted by Theorem~\ref{theorem2}, the slope of the resulting curve is larger near the left endpoint of the instability region than near the right endpoint. Further away from these endpoints, the solution shapes change and we conjecture a period doubling bifurcation near $\Delta T = 0.0646$, rapidly followed by a torus bifurcation. This would be consistent with the sharper increase in the sup-norms in Figure~\ref{fig:fig_nonlinear_attractor}, and the fact that the periods of the solutions become rather large, see Figure~\ref{fig:fig_traces}(b). Despite this detour to another attractor, the solution is eventually turning into the near harmonic periodic solution bifurcating from the right endpoint $\Delta T_2$, see Figure~\ref{fig:fig_plots} (a).

\begin{figure}[H]
\begin{center}
\begin{tabular}{cc}
	\includegraphics[scale = 0.32]{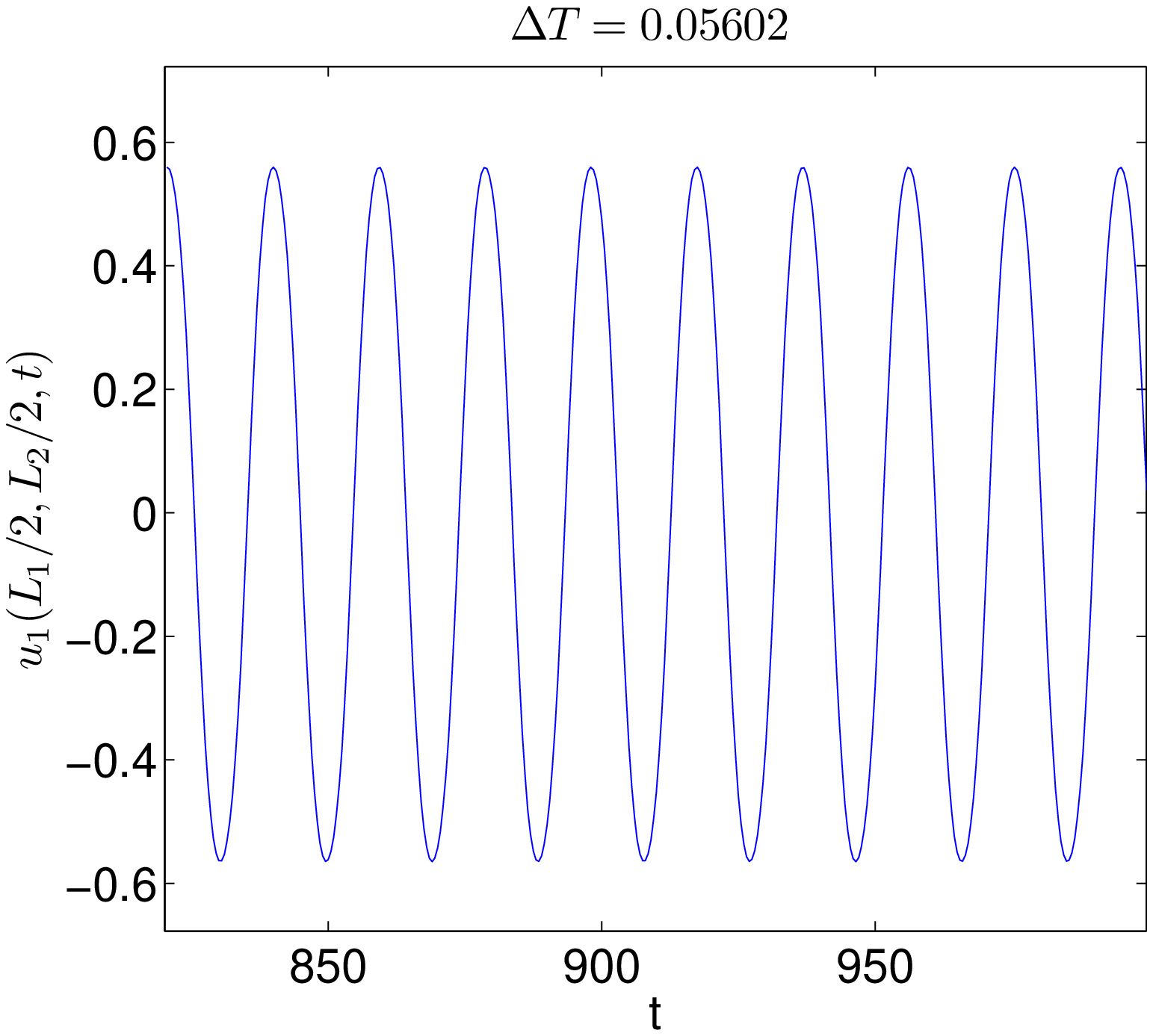}
&	\includegraphics[scale = 0.32]{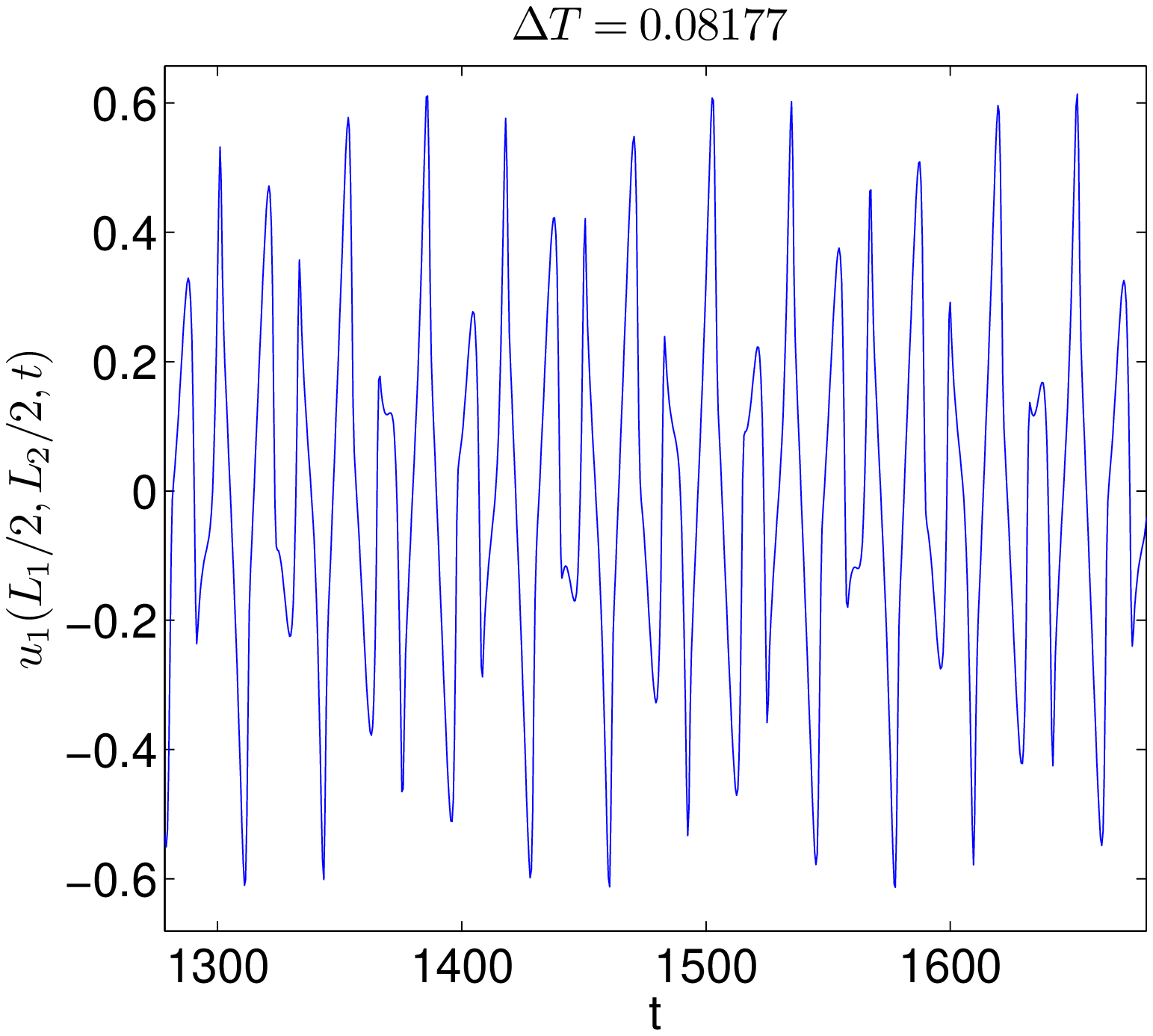}\\
(a) & (b)
\end{tabular}
	\caption {Time traces of $u_1$ at the midpoint $(x_1,x_2)=(L_1/2,L_2/2)$. (a) $\Delta T = 0.056$, and (b) $\Delta T = 0.08$. Compare Figure~\ref{fig:fig_nonlinear_attractor}.}
	\label{fig:fig_traces}
\end{center}
\end{figure}

As an example for instabilities caused by higher spatial harmonics, we plot in Figure~\ref{fig:fig_plots2} a solution that emerged from an instability with wavenumber $k_2=2$.

\begin{figure}[H]
\begin{center}
\includegraphics[scale = 0.35]{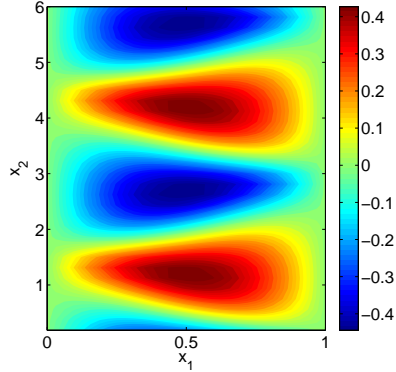}
\caption{Contour plot of $u_1(t_2)$ inside the nonlinear regime. In this case, the critical eigenfunction has a wavenumber $k_2 = 2$.}
\label{fig:fig_plots2}
\end{center}
\end{figure}

%%%%%%%%%%%%%%%%%%%%%%%%
\section{Reformulation and setting}
\label{refor}
For the bifurcation study it is convenient to formulate \eqref{eq1} through the deviation $u=(u_1,u_2)$ from $\rho_{ss}$, 
\begin{eqnarray*}
\rho^{+} = \uP + \rho_{ss}^{+}, \rho^{-} = \uM + \rho_{ss}^{-}.
\end{eqnarray*}
In terms of $u$, and in the comoving variable $x_2\to x_2+T^- t$, system \eqref{eq1} reads
\begin{equation}\label{eq2}
\left\{
\begin{aligned}
\partial_t \uP &= \DT \partial_{x_2} \uP + E_2/L_1 - E^{\perp} \cdot \nabla \uP + \nu \nabla^2 \uP,\\
\partial_t \uM &= \hspace{16mm} - E_2/L_1 - E^{\perp} \cdot \nabla \uM + \nu \nabla^2 \uM,\\
E &=-\nabla V,\\
-\nabla^2 V &= \uP + \uM,\\
x &\in [0, L_{1}] \times \R/L_2 \Z\mbox{ ,}t \geq 0,
\end{aligned}
\right.
\end{equation}
subject to (periodic b.c.\ in $x_2$ and) homogeneous Dirichlet boundary conditions
\begin{equation}
\begin{aligned}
\uP(0,x_2,t) &= \uM(0,x_2,t) = V(0,x_2,t) = 0,\\
\uP(L_1,x_2,t) &= \uM(L_1,x_2,t) = V(L_1,x_2,t) = 0.
\end{aligned}
\end{equation}

\begin{remark}
We want to briefly point out a peculiarity of the nonlinearity in  \eqref{eq1} and equivalently \eqref{eq2}: viewed on complexified phase space, each eigenspace of the laplacian is flow invariant and the dynamics is purely linear.

Indeed, take an eigenfunction $e$ with eigenvalue $\lambda$ and set $u_j= \alpha_j e$ with $\alpha_j\in\C$ so that $E= (\alpha_1+\alpha_2)/\lambda \nabla e$. Hence, $E^\perp\cdot\nabla u_j =0$ so that \eqref{eq2} is in fact linear.

However, this does not provide flow invariant spaces for the real equations since all eigenvalues and eigenspaces are complex, and the previous argument is incorrect for linear combinations. Co-moving frames do not generate real eigenspaces due to the asymmetric advection terms. 
\end{remark}

Next we choose a simple functional analytic setting for a formulation of \eqref{eq2} as a parabolic problem by solving the Poisson equation. This is convenient for the center manifold reduction, but also gives a simple well-posedness setting. 

Let $\Omega := [0, L_1] \times [0, L_2]$ and denote the Sobolev spaces $\Hspace^j = H^j([0, L_{1}]\times \R/L_2 \Z)$ as well as 
\begin{equation}
\begin{aligned}
&X :=   H^1_0([0, L_{1}]\times \R/L_2 \Z),\\
&Y := \{ f \in \Hspace^2\,:\, f(0,x_2) = f(L_1,x_2) = 0\},  \\
&Z := \{ f \in \Hspace^3\,:\, f(0,x_2) = f(L_1,x_2) = 0\},\\
\end{aligned}
\end{equation}
which incorporate the Dirichlet boundary conditions.
We shall use standard notation: for $f_1$, $f_2$ $\in \Lspace^2([0,L_1]\times\mathbb{R}/L_2 Z)$ we denote the scalar product by $\langle f_1, f_2 \rangle = \int_{\Omega} f_1(x) \overline{f_2(x)} dx$ and for $\mathrm{f}_j = (\mathrm{f}_{j,1}, \mathrm{f}_{j,2}) \in L^2([0,L_1]\times\mathbb{R}/L_2 Z)^2$ $j = 1, 2$ by $\langle \mathrm{f}_1, \mathrm{f}_2 \rangle_{2} = \langle \mathrm{f}_{1,1} , \mathrm{f}_{2,1}\rangle +  \langle \mathrm{f}_{1,2} , \mathrm{f}_{2,2}\rangle$.

Thanks to these boundary conditions, we can solve the Poisson equation in \eqref{eq2}; see also \S\ref{s:spec} for explicit solutions. We thus obtain $E$ via the bounded operators $A_j: Z \to \Hspace^3$ defined by
\begin{equation}\label{defA}
\begin{aligned}
A_j f &:= \partial_{x_{j}}( \nabla^2)^{-1} f, \; j = 1,2,\\
A f &= (A_1 f, A_2 f)^{T},\\
A^{\perp} f &= (A_2 f, -A_1 f)^{T}. 
\end{aligned}
\end{equation}
Notably, $A_2$ in fact maps into $Z$, because $E_2 = \partial_{x_2} V$ vanishes for $x_1=0, L_1$ due to the Dirichlet boundary conditions. 

In order to apply standard results on parabolic equations, let us write \eqref{eq2} equivalently in the standard form
\begin{equation}\label{eq3}
\frac{ d u}{dt} = \L u + R(u),
\end{equation}
so that solutions of this and \eqref{eq1} are in 1-to-1 correspondence. Here
\begin{equation*}
\begin{aligned}
\L u &= \begin{pmatrix} 
	\DT \partial_{x_2} \uP  + \frac{1}{L_1} A_2(\uP + \uM) + \nu \nabla^2 \uP\\
	\hspace{14mm}- \frac{1}{L_1} A_2(\uP + \uM) + \nu \nabla^2 \uM 
\end{pmatrix}, \\
R(u) &= \begin{pmatrix}
 -A^{\perp}(\uP + \uM) \cdot \nabla \uP\\
 -A^{\perp}(u_1+u_2) \cdot \nabla u_2
\end{pmatrix}.
\end{aligned}
\end{equation*}

Note that $\L \in \mathcal{L}(Z \times Z,X \times X)$ is the linearization of \eqref{eq1} in $\rho_{ss}$. We have that $R: Z \times Z\to Y \times Y$ since $\nabla u_j$ vanish at $x_1=0,L_1$ and $\Hspace^2$ is a Banach algebra; $R$ is in fact analytic in $u$. See also \S \ref{s:center}. Moreover, the imbeddings $Z^2 \hookrightarrow Y^2 \hookrightarrow X^2$ are dense and the uniformly elliptic operator $-\L:Z\times Z\subset X \times X\to X \times X$ is a sectorial operator, generating an analytic semigroup, and so \eqref{eq2} admits mild and classical solutions $u(t)$ for any initial condition $u(0)\in Y \times Y$. The sectoriality is a consequence of the fact that the laplacian is sectorial in $Y$ with domain $\Lspace^2$ of the cylinder \cite{Henry}, and this is robust under addition of the lower order terms in $\L$. It thus also possesses a square root, which then provides an isomorphism from $\Lspace^2$ to $X$. Hence, $\L$ is also sectorial on $Z$ with domain $X$. Note also that  $\L$ has a compact resolvent and thus discrete spectrum accumulating at $-\infty$. We discuss its spectrum in detail in the next section. 

%%%%%%%%%%%%%%%%%%%%%%%%
\section{Spectrum of the linearization}\label{s:spec} 

For the bifurcation analysis, we distinguish the stable spectrum of $\L$, $\sigma_{-}(\L):=\{ \lambda \in \sigma(\L) : \Re \lambda < 0 \}$, its neutral spectrum $\sigma_{0}(\L)=\{ \lambda \in \sigma(\L) : \Re \lambda = 0 \}$ and its unstable spectrum $\sigma_{+}(\L)=\{ \lambda \in \sigma(\L) : \Re \lambda > 0 \}$. 

The next Lemma characterizes the spectrum and is the basis for the identification of bifurcations. While this concerns the comoving variable of system \eqref{eq2}, the spectrum for the original system is the same up to a scaling of the imaginary parts. See \S\ref{s:tw}.

\begin{lemma} \label{l:spec}
The spectrum $\sigma(\L)$ of $\L$ consists of the eigenvalues 
\begin{equation}\label{e:lam}
\lambda_k^\pm = \rmi \pi\frac{k_2 \DT}{\ell L_1}  - \pi^2 \frac\nu {L_1^2}\left( k_1^2 + \frac{4 k_2^2}{\ell^2} \right) \pm \sqrt{\ds}, \; k \in \N \times \Z,
\end{equation}
where $\N=\NN\setminus\{0\}$ and 
\begin{equation}\label{discr}
\begin{aligned}
\ds = \frac{k_2^2\DT}{\ell^2 L_1} \left( \frac{4}{ k_1^2 + 4 (k_2/\ell)^2} - \pi^2 \frac \DT{L_1}\right).
\end{aligned}
\end{equation}
In particular, $\lambda_k^-\in \sigma_-(\L)$, and if $\ds\leq 0$ then $\lambda_k^+\in \sigma_-(\L)$. 
Moreover, $\Re\left(\lambda_{(k_1,k_2)}^+\right)< \Re\left(\lambda_{(1,k_2)}^+\right)$.
\end{lemma}

We will start to discuss the relevance and implications of this result after the proof.
In preparation of the proof, choose the orthogonal basis of $X$ given by
\begin{equation}\label{e:gk}
g_k(x) := \sin\bigg(\frac{k_1 \pi x_1}{L_1}\bigg) e^{\frac{2 \rmi \pi k_2 x_2}{L_2}},
\end{equation}
where $k \in \N \times \Z$.
In order to express the operator $A$, denote
\begin{equation}
\phi_k(x) := \cos\bigg(\frac{k_1 \pi x_1}{L_1}\bigg) e^{\frac{2 \rmi \pi k_2 x_2}{L_2}}.
\end{equation}
Indeed, if $f \in X$, the explicit solution to the Poisson equation $-\nabla^2 V = f$ in terms of this basis reads
\begin{eqnarray*}
V(x) =  \frac{2}{\pi^2} \sum_{k \in \N \times \Z} \frac{1}{ \frac{L_2}{L_1} k_1^2 + 4 \frac{L_1}{L_2} k_2^2}
\left(\int_{\Omega}f(y)\overline{g_k(y)}dy\right) \, g_k (x).
\end{eqnarray*}
We therefore get the explicit formula for $A$:
\begin{equation}\label{e:Agk}
A f(x) = -\frac{2}{\pi} \sum_{k \in \N \times \Z} \langle f, g_k \rangle \bigg( \frac{L_2}{L_1}k_1^2 + \frac{4 L_1}{L_2} k_2^2 \bigg)^{-1}
\begin{pmatrix}
k_1 \phi_k(x)/L_1 \\
2 \rmi k_2 g_k(x) / L_2
\end{pmatrix}.
\end{equation}

\begin{proof}[Lemma~\ref{l:spec}]
Consider functions of the form $q g_k(x)$, $k \in \N\times \Z$, where $q \in \C^2$ is an arbitrary constant vector. Since
\begin{align}\label{e:A2gk}
A_2 g_k(x) = -\frac{2 L_1 \rmi}{\pi}\frac{k_2}{\frac{L_2}{L_1}k_1^2 + 4 \frac{L_1}{L_2}k_2^2}g_k(x),
\end{align}
the action of $\L$ on such functions is
\begin{eqnarray}\label{eq13}
(L q g_k )(x) = M_k q g_k(x),
\end{eqnarray}
where
\begin{equation}
M_k :=
\begin{pmatrix}
C_1(k) \DT - C_2(k) - C_3(k) & -C_2(k) \\
C_2(k) & C_2(k) - C_3(k)
\end{pmatrix},\label{eq16}
\end{equation}
with
\begin{equation*}
C_1(k) := \frac{2 \pi k_2 \rmi}{L_2}\mbox{, }C_2(k) := \frac{2 \rmi}{\pi}\frac{k_2}{\frac{L_2}{L_1}k_1^2 + \frac{4 L_1}{L_2} k_2^2}\mbox{, } C_3(k) := \nu \pi^2 \bigg( \frac{k_1^2}{L_1^2} + \frac{4 k_2^2}{L_2^2}\bigg).
\end{equation*}.

The eigenvalues of $M_k$ are readily computed to be $\lambda_k^\pm$. The claims on the real parts of $\lambda_k^\pm$ immediately follow from inspecting \eqref{e:lam} -- in particular $\ds$ monotonically decreases in $k_1$. 
\end{proof}

Note that the proof also implies that eigenfunctions of $\L$ have the form 
\begin{equation}\label{e:zeta}
\zeta_{k}(x) := \xi_{k} g_k(x)\in Z \times Z,
\end{equation}
with $\xi_k$ a eigenvector of $M_k$.

\medskip
The last statement in Lemma~\ref{l:spec} means that only $\ds>0$ and $\lambda_{(1,k_2)}^+$ with $k_2\in\Z\setminus\{0\}$ allow for destabilization, and the real part in this case is given by
\begin{equation}\label{e:relam}
\Re\left(\lambda_{(1,k_2)}^+\right) = - \pi^2 \frac\nu {L_1^2}\left( 1 + \frac{4 k_2^2}{\ell^2} \right) + \sqrt{\frac{k_2^2\DT}{\ell^2 L_1} \left( \frac{4}{1 + 4 (k_2/\ell)^2} - \pi^2 \frac \DT{L_1}\right)}.
\end{equation}
Note that this is a function of the three parameters $\nu/L_1^2$, $\DT/L_1$, $(k_2/\ell)^2$. As expected, increasing viscosity always stabilizes, with increasing impact for increasing $(k_2/\ell)^2$. However, the dependence of the real part on $k_2/\ell$ is not necessarily monotone, which allows for intricate destabilization scenarios.

The imaginary part, $\Im(\lambda_{(1,k_2)}^+)$, is never zero, which means that all bifurcations are non-stationary and we generically expect Andronov-Hopf bifurcations, where $k_2$ determines the wavenumber of bifurcating solutions. 

We consider the temperature difference $\DT$ as the primary bifurcation parameter and therefore focus on the location of instabilities as $\DT$ varies, as well as on the wavenumber of destabilizing modes determined by $k_2$.

\begin{figure}[H]
\begin{center}
\begin{tabular}{cc}
\includegraphics[scale = 1.2]{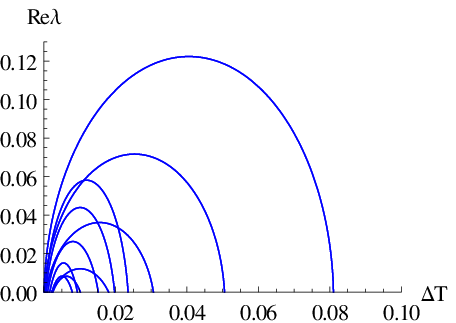}\hspace{5mm}
& \includegraphics[scale = 1.2]{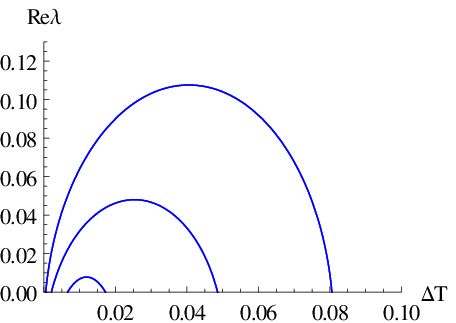}\\
(a) & (b)
\end{tabular}
\caption{\label{fig:fig2} Real parts of eigenvalues as functions of $\DT$ for $L_1 = L_2 = 1$. The unstable eigenvalues with $k \in \lbrace 1, ..., 10\rbrace \times \lbrace-10, ..., 10 \rbrace$ are plotted for (a) $\nu = 10^{-3}$, (b) $\nu = 4\cdot 10^{-3}$.}
\end{center}
\end{figure}

In Figure~\ref{fig:fig2} we plot sample computations of spectrum as $\Delta T$ varies, illustrating the stabilizing effect of the viscosity. Crossings of eigenvalue curves at zero real part can occur, which is expected to generate rich bifurcations. However, in this paper we focus on simple Andronov-Hopf bifurcations.

\medskip
Recall the spectral conditions at a primary Andronov-Hopf bifurcation 
\begin{equation}\label{eq5}
\begin{aligned}
&\mbox{(i) }\mbox{There is a constant }\gamma>0\mbox{ s.t.}\sup\{\Re \lambda : \lambda \in \sigma_{-}(\L) \}  <-\gamma,\\
&\mbox{(ii) } \sigma_{0}(\L) = \{ \pm \rmi \omega \} , \omega>0\mbox{ and } \pm \rmi \omega \mbox{ are simple eigenvalues},\\
&\mbox{(iii) } \sigma_{+}(\L) = \emptyset,
\end{aligned}
\end{equation}
and in the nondegenerate case, the critical eigenvalues transversely cross the imaginary axis upon parameter variation.

\medskip
It turns out that we can characterize a large part of parameter space, where critical eigenvalues have $k_2=1$, that is, $k=k_c:=(1,1)$. We therefore define the following particular case of \eqref{eq5}.
 \begin{hypothesis} 
\label{hypoth1}
It holds that $\Re \lambda^{+}_{k_c} = \Re \lambda^{+}_{\overline{k}_c} = 0$ and there is $\gamma>0$ such that $\Re \lambda^{\pm}_k< -\gamma$ for $k \in \N \times \Z \setminus \{k_c,\overline{k}_c\}$.
\end{hypothesis}

Here and in the following we denote $\overline{\kappa} = (\kappa_1,-\kappa_2)$ for $\kappa \in \R^2$.

Rearranging sign conditions on \eqref{e:relam} and squaring, we readily compute that the sign of $\Re(\lambda_{(1,k_2)}^+)$, for $\kappa_2=k_2^2$ is the sign of
\begin{equation}
\label{fun1red}
\d(\DT,\kappa_2) = \frac{4 L_1^3}{4\kappa_2+\ell^2} \DT - \frac{L_1^2\pi^2}{\ell^2}\DT^2 - \nu^2\pi^4\frac{(4\kappa_2+\ell^2)^2}{\kappa_2 \ell^4},
\end{equation}
which is somewhat simpler to handle. In particular, zeros of $\d$ are the critical eigenvalues for bifurcations. This yields the following a priori bounds on $\DT$ for linear instability.

\begin{lemma}\label{l:bound}
For all $\nu, \ell$ and $\kappa_2>0$, the real roots of $\d(\cdot,\kappa_2)$ lie in $[0, 4L_1/\pi^2]$. Moreover, the real roots approach the endpoints in the limit $\ell\to \infty$ if $\nu=o(\ell^{-1})$.
\end{lemma}

\begin{proof}
Since $\d(0,\kappa_2)\leq 0$ and $\partial_{\DT} \d (0,\kappa_2) > 0$ the lower bound holds. For the upper bound, observe that $\d(4L_1/\pi^2,\kappa_2)<0$ and $\partial_{\DT}d(4 L_1/\pi^2,\kappa_2)<0$, which proves the claim since the quadratic coefficient of $\DT$ is negative. The statement on the limits readily follows from \eqref{fun1red} upon multiplication by $\ell^2$.
\end{proof}

Note that $\d(\cdot,\kappa_2)$, as a quadratic polynomial in $\Delta T$, has two real roots $\Delta T_1(\kappa_2) \leq \Delta T_2(\kappa_2)$ if and only if the viscosity is sufficiently small,
\begin{equation}
\label{inst_condred}
\nu \leq \nu_\crit(\kappa_2):=\frac{2 \sqrt{\kappa_2} \ell^3 L_1^2}{(4 \kappa_2 + \ell^2)^2 \pi^3},
\end{equation}
with a double root at equality.  Hence, this is a necessary and sufficient condition for the occurrence of critical eigenvalues $\lambda_{(1,\sqrt{\kappa_2})}^+$ as $\DT$ varies. However, it is subtle to determine when the critical eigenvalues destabilize the equilibrium as this requires to exclude unstable eigenvalues for all other $k_2$.

Nevertheless, the location of these parabola's maxima in $\DT$ is at
\begin{equation}\label{e:maxloc}
\DT = \frac{2 \ell^2 L_1}{(4 \kappa_2 + \ell^2) \pi^2},
\end{equation}
which is strictly decreasing in $\kappa_2$. Therefore, the $k_2$-value of these parabola in $\DT$ can be identified by the relative location of their maxima.\\

\begin{remark}\label{r:thres}
For $\kappa_2=1$ the roots satisfy $\Delta T_1=\calO( \nu^2 )$ and $\Delta T_2=\frac{4 \ell^2 L_1}{\pi^2(\ell^2+4)}+\calO(\nu^2)$, which was already illustrated in Figure~\ref{fig:fig1}.
\end{remark} 

The geometric nature of bifurcating solutions is determined by the $k_2$-value of critical and destabilizing eigenvalues as $\DT$ in- or decreases from outside $[0,4L_1/\pi^2]$. We thus define

\begin{definition}
For given $L_1, \ell$, $\nu$, we say that $\L$ possesses a $k_2$-instability region, if $\d(\cdot,k_2)$ has two positive roots $\Delta T_1(k_2^2) \leq \Delta T_2(k_2^2)$. We call a $k_2$-instability region locally primary, if there is a neighbourhood $S \subseteq \mathbb{R}$  of $J:= (\Delta T_1(k_2^2), \Delta T_2(k_2^2))$, s.t. the steady state $u=0$ is stable for $\Delta T \in S\setminus J$  and $\DT_j(k_2^2)\neq \DT_j(\kappa_2)$ for $\kappa_2\neq k_2^2$, $j=1,2$. Moreover, we say that the $k_2$-instability region is primary, if it is locally primary and $S = \mathbb{R}$.
\end{definition}

To ease notation, we simply write $\Delta T_j$ for $\DT_j(1)$, $j=1,2$.

As a first step to understand the nature of destabilizing $k_2$-instability regions, we consider the case $k_2=1$ and in preparation define the following condition.

\begin{hypothesis}\label{h:prima}
Suppose that for given $L_1, \ell, \nu>0$ we have
\begin{equation}
\label{k2_1_primary_cond}
\frac{\DT}{\nu^2 \pi^4} > \frac{(4 + \ell^2) (4 k_2^2 + \ell^2)}{16\ell^4 L_1^3 k_2^2} (\ell^4-16 k_2^2)
\end{equation}
and $\nu < \nu_{crit}(1)$ for $\Delta T\in \{\Delta T_1,\Delta T_2\}$, and all $k_2\in\N$, $k_2\geq 2$.
\end{hypothesis}

Note that Hypothesis~\ref{h:prima} requires a ratio of temperature difference and viscosity to dominate a ratio involving domain geometry and linear mode harmonics.

\begin{theorem}
\label{theorem1}
\begin{enumerate}
\item A $1$-instability region of $\L$ is locally primary if and only if Hypothesis~\ref{h:prima} holds. If it holds, then Hypothesis~\ref{hypoth1} is satisfied at $\DT=\Delta T_j$, $j=1,2$.
The critical eigenvalues are $\lambda_j = \pm\rmi\omega_j$ with $\omega_j = \pi\DT_j/(\ell L_1)$.
\item For $0<\ell \leq 2\sqrt{2}\approx 2.8$ any $1$-instability region of $\L$ is primary and Hypothesis~\ref{hypoth1} is satisfied at  $\DT=\Delta T_j$, $j=1,2$.
\end{enumerate}
\end{theorem}

The point of the theorem is that it provides conditions (Hypothesis~\ref{h:prima} or $\ell\leq 2\sqrt{2}$) under which the destabilizing mode for increasing and decreasing $\DT$ is known, namely the lowest spatial harmonic. Note that the values of $\ell$ in particular include the case $\ell=1$ considered in \cite{daniel1}.

\begin{proof}
\begin{enumerate}
\item A direct calculation gives 
\begin{align}\label{e:ddiff}
\frac{\d(\DT,1)-\d(\DT,\kappa_2)}{\kappa_2 - 1} &= \frac{16 \DT \kappa_2 \ell^4 L_1^3 - (4 + \ell^2) (4 \kappa_2 + \ell^2) (\ell^4-16 \kappa_2) \nu^2 \pi^4}{\kappa_2 \ell^4 (4 + \ell^2) (4 \kappa_2 + \ell^2)}.
\end{align}
In particular, Hypothesis~\ref{h:prima} is indeed equivalent to a $1$-instability region being locally primary. The claims on $\DT$ follow readily from inspection of the zeros of $\d$. 
\item This is the trivial observation that the right hand side in condition \ref{k2_1_primary_cond} is strictly negative for these values of $\ell$, while the left hand side is positive at all possible real roots $\Delta T$ of $d(\cdot, \kappa_2)$ on account of Lemma~\ref{l:bound}.
\end{enumerate}
\end{proof}

\begin{remark}
The critical frequencies in the original $x_2$-variable of \eqref{eq1} are in fact
\[
\omega_j = \frac{T^+_j + T^-_j}{L_2}\pi.
\]
\end{remark}

Now, we are going to present a condition, which guarantees that other destabilization scenarios also occur.
\begin{corollary}\label{c:explicit}
Let $\kappa_2>1$ and let $\ell$ be the unique positive solution $\ell=\ell_{\kappa_2}$ of 
\begin{equation}\label{e:counter}
\ell^6 - \kappa_2 \ell^4 - 80\kappa_2 \ell^2 - 64 \kappa_2 (2 + \kappa_2)=0.
\end{equation}
Then for $\nu=\nu_\crit(1)$ the $1$-instability region is a point, $\DT_1=\DT_2$, that coincides with $\DT_2(\kappa_2)$.  Notably, $\ell_{\kappa_2}$ is strictly increasing in $\kappa_2$.
\end{corollary}

This means that the $1$-instability region is not primary. In fact, it is also not primary for nearby parameter values that produce $\DT_2(4)>\DT_2(1)$. The solution to \eqref{e:counter} for $\kappa_2=4$ is $\ell_4\approx 5.37$, and for $\kappa_2=9$ it is $\ell_9\approx 7.22$. See Figure~\ref{f:counter}. For $\ell$ between these value (and slightly above $\ell_9$), we numerically find that the $2$-instability region is primary. We omit the tedious analysis.
In general, for any given $k_2$ Hypothesis~\ref{h:prima} is violated for sufficiently large $\ell$ (with $\nu,L_1$ fixed), since $\DT$ is bounded (Lemma~\ref{l:bound}).

\begin{remark}
It is possible to show that for $\nu$ small enough, there is a primary 1-instability region, if $\ell < \ell^{\ast} \approx 4.053$, where $\ell^{\ast}$ is the unique positive root of the polynomial $16(\ell^2 + 4)^2 - (\ell^2 - 8) (\ell^2 + 8) (\ell^2 + 16)=0$. %(see Appendix)%
\end{remark}

\begin{figure}[H]
\begin{center}
\begin{tabular}{cc}
\includegraphics[scale = 0.55]{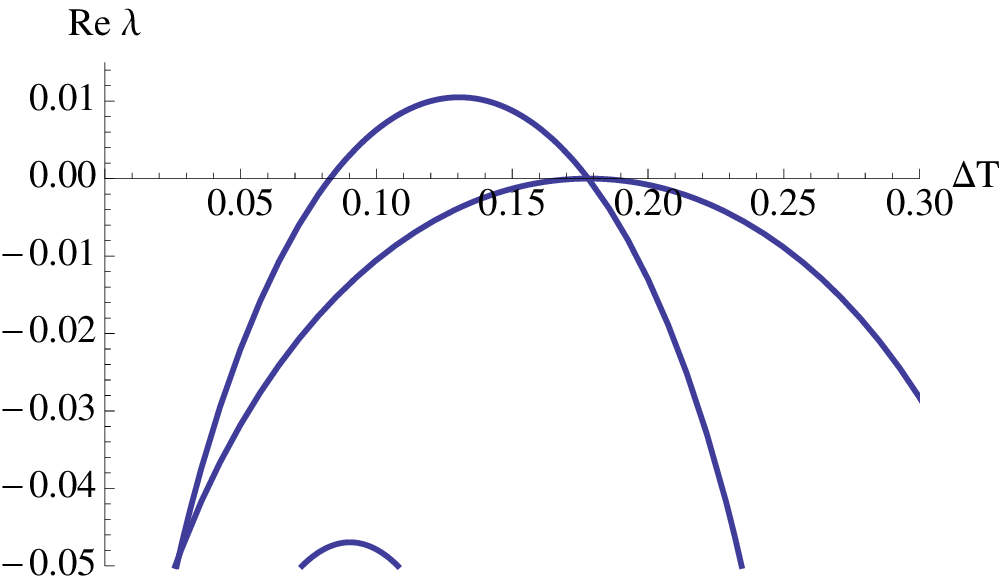}\hspace{5mm}
& \includegraphics[scale = 0.55]{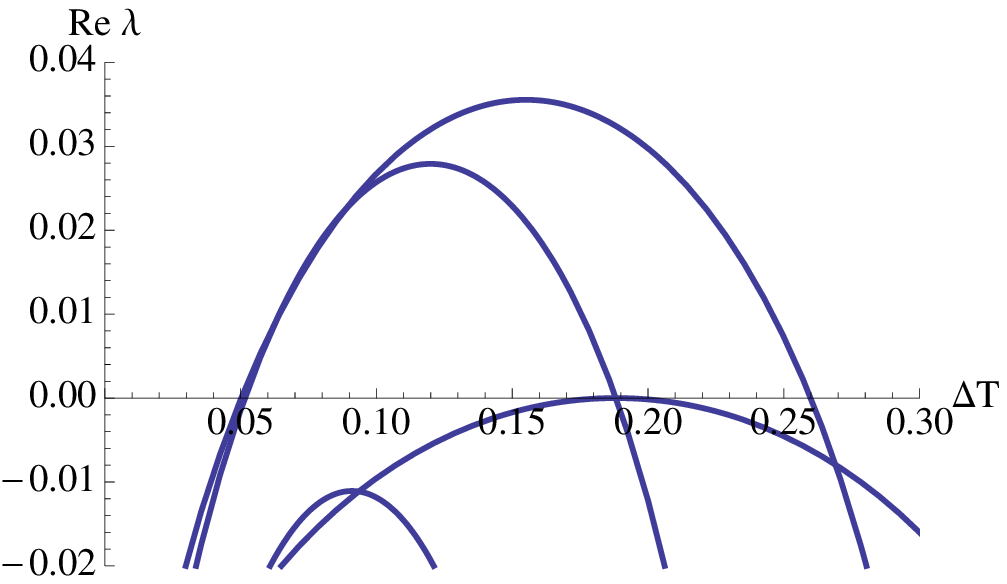}\\
(a) & (b)
\end{tabular}
\caption{\label{f:counter}Real parts of eigenvalues $\lambda_k^+$, $k=(1,k_2)$, as functions of $\DT$. The parabola are ordered in $k_2$ by the decreasing location of maxima. Here $\nu=\nu_\crit(1)$ so that $\DT_1=\DT_2$. (a) $\ell=\ell_4$, with near primary $2$-instability region. (b) $\ell=\ell_9$, with primary $2$-instability region.}
\end{center}
\end{figure}

\begin{proof}[Corollary~\ref{c:explicit}]
Substituting the critical $\nu^2= \frac{4 l^6 L_1^4}{(4+l^2)^4 \pi^6}$ from \eqref{inst_condred} and the corresponding critical value of $\DT = \frac{2 \ell^2 L_1}{\pi^2 (4+\ell^2)}$ at the double root into the nominator of the right hand side of \ref{e:ddiff} gives
\[
\frac{4 \ell^6 L_1^4}{(4 + 
   \ell^2)^3 \pi^2}(64 \kappa_2 (2 + \kappa_2) + 80 \kappa_2 \ell^2 + 4 \kappa_2 \ell^4 - \ell^6),
\]
where $\kappa_2=k_2^2$. The first factor is positive and roots of the second factor, which we denote by $q$, precisely solve \eqref{e:counter}. We have
\[
\partial_{(\ell^2)} q = 80 \kappa_2 + 8 \kappa_2 \ell^2 - 3 \ell^4,
\]
which is positive at $\ell=0$ so that the cubic $q$ with negative cubic coefficient has a unique positive root. In addition, this implies that $\partial_\ell q<0$ at this root so that together with
\[
\partial_{\kappa_2} q = 4 (32 + 32 \kappa_2 + 20 \ell^2 + \ell^4) > 0
\]
we infer from implicit differentiation that the location of this root strictly increases with $\kappa_2$.
\end{proof}

\medskip
For the case of small viscosity (and $\ell>2.8$), we omit the somewhat tedious detailed analysis for the destabilizing left endpoint. However, we immediately obtain the following.

\begin{corollary}\label{c:nusmall}
As $\nu\to 0$, $\L$ has $k_2$-instability regions for $k_2\to \infty$ with $\Delta T_1(k_2^2)<\Delta T_2(k_2^2)$. For sufficiently small $\nu$, the conditions of Hypothesis~\ref{hypoth1} are satisfied at $\DT_2(1)$, and this is an instability threshold.
\end{corollary}

\begin{proof}
The presence of all $k_2$-instability regions clearly holds at $\nu=0$ in view of \eqref{inst_condred}. In addition, from \eqref{e:ddiff} we infer at $\nu=0$ that
\[
\d(\DT,1)-\d(\DT,\kappa_2) > 0,
\]
so that the critical eigenfunction at the right endpoint of the instability interval has mode number $k_2 = \pm 1$. This persists for sufficiently small $\nu > 0$, since the thresholds depend continuously on $\nu$, and again from \ref{e:ddiff} we see that for each $\nu>0$ there is only a finite range of $\kappa_2$ values, for which $d(\Delta T,1) - d(\Delta T,\kappa_2) < 0$ is possible. 
\end{proof}

Lastly, we point out the possibility of multiple disjoint primary $k_2$-instability regions, where changing $\DT$ destabilizes and stabilizes multiple times. In  Figure~\ref{f:mult} we plot eigenvalue curves, where two $k_2$-instability regions consist of a point. Parameters $\nu=\nu_\crit(k_2)=\nu_\crit(k_2')$ and $\ell$ that produce such scenarios can be readily computed from \eqref{inst_condred}; here we take $k_2=1$, $k_2'=4$.  For perturbed $\nu<\nu_\crit(1)$ the instability regions become disjoint open intervals.

\begin{figure}[H]
\begin{center}\includegraphics[scale = 0.55]{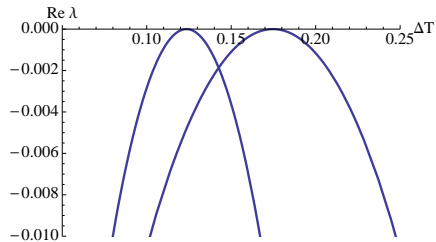}
\caption{\label{f:mult}Real parts of eigenvalues $\lambda_k^+$, $k=(1,k_2)$, as functions of $\DT$. Here $\ell=2 \sqrt{2 + 3 \sqrt{2}}$ so that $\nu_\crit(1)=\nu_\crit(4)$. In the region plotted only the curves with $k_2=1,4$ are present.}
\end{center}
\end{figure}

\begin{remark}\label{r:sndHopf}
On account of \eqref{inst_condred}, for decreasing $\nu$ and also for increasing $\ell$, there is an increasingly long sequence of secondary instabilities of Andronov-Hopf type as $\Delta T$ increases from zero, with higher and higher spatial harmonics, and another reverse sequence as $\Delta T$ reaches $\Delta T_*$. See Figure~\ref{fig:fig2}.
\end{remark}

%%%%%%%%%%%%%%%%%%%%%%%%%%%
\section{Center manifold reduction}\label{s:center}

In this section, we consider the vicinity of parameters with critical $\DT = \DT_j$ for $j=1$ or $j=2$ and assume that no other eigenvalues lies on the imaginary axis. The main example is a primary $1$-instability region. For the unfolding of the bifurcation in the generic case $\DT_1 < \DT_2$ we introduce the parameter $\mu_1$ by $\DT = \DT_j + \mu_1$. In the degenerate case $\DT_1 = \DT_2$, where $\nu=\nu_\crit(1)$, we additionally unfold with $\mu_2$ defined by $\nu = \nu_\crit(1) - \mu_2^2$.
For readability we frequently suppress the index $j$. 

At bifurcation, the critical eigenvalues are then $\pm \rmi \omega$ and we denote the associated eigenfunctions by $\zeta(x):=\zeta_{k_c}(x)$, $\overline{\zeta}(x)$, see \eqref{e:zeta}. Then $\L$ possesses a two-dimensional real central subspace $E_c:=\mathrm{span}\{\Re \zeta, \Im \zeta\} \subset Z^2$ and we will show that there is a locally invariant 2D center manifold
\[
W_c = \{ u_0 + \psi(u_0,\mu): u_0 \in O_{E_c} \} \subset O_{Z^2}\mbox{ , }\mu \in O_{\R^2}, 
\]
with $\psi:O_{E_c}\to E_c^\sharp$, $E_c\oplus E_c^\sharp=Z^2$, and neighbourhoods $O_{\R^2}$ of $\mu=0$, and $O_{E_c}$, $O_{Z^2}$ of $0\in Z^2$. In case of a primary bifurcation the center manifold is also locally exponentially attracting.

Since we consider $k=k_c=(1,1)$, it is not surprising that the coefficients $C_m(k_c)$ defined in \eqref{eq16} show up. It turns out that following modifications are convenient.
\begin{equation}
c_1 := \frac{\pi \Delta T}{L_2} \mbox{ , }c_2 := \frac{2}{\pi\big( \frac{L_2}{L_1} + \frac{4 L_1}{L_2} \big)} \mbox{ , }
c_3 := \nu \pi^2 \bigg( \frac{1}{L_1^2} + \frac{4}{L_2^2} \bigg). \label{eq_const}
\end{equation}

We first consider the generic case of \eqref{inst_condred}, where the unfolding goes by $\mu_1$ only. 

\begin{theorem}\label{theorem2}
Assume that Hypothesis~\ref{hypoth1} holds for a fixed parameter set for which $\DT_1<\DT_2$. Then the steady state $u=0$ of system \eqref{eq3} possesses a locally exponentially attracting and locally invariant 2D center manifold near $u=0$ with the reduced dynamics
\begin{equation}\label{eq9}
\frac{d z} {dt} = \rmi \omega z + \mu_1 a z + b z |z|^2 +  \mathcal{O}(|z|(|\mu|+|z|^2)^2),
\end{equation}
where
\begin{equation}
\begin{aligned}
\omega &= \frac{\pi\DT}{L_2},\\
a &= \frac{2 \pi c_2^2}{L_2 (c_1 - c_3\rmi) (c_1\rmi - 2 c_2\rmi -c_3)},\\
b &= - \frac{L_1^3}{4 \pi^2 \nu}\frac{c_1^2 + c_3^2}{\frac{L_2^2}{L_1}+4 L_1}.
\label{eq_norm_form}
\end{aligned}
\end{equation}
\end{theorem}

The following corollary proves the nature of the resulting bifurcations, see Fig. \ref{fig:fig1}(a) for an illustration.

\begin{corollary}
\label{c:hopf}
Assume the Hypotheses of Theorem~\ref{theorem2}. Then the steady state $u=0$ of system \eqref{eq3} undergoes a generic supercritical Andronov-Hopf bifurcations as $\mu_1$ varies. Specifically, the reduced vector field coefficients satisfy $b<0$, $\Im(a)=\calO(\nu^2)$, and $\mathrm{sgn}(\Re a) = -(-1)^j$ at $\Delta T= \Delta T_j$. 

In particular, near the stability thresholds there exist heteroclinic connections between the unstable steady-state and the stable limit cycle. 

As $\nu \downarrow 0$, the radius of the limit cycles, $|z(t)|$, scales near $\Delta T_1$ as $|z(t)| \propto \nu^{-1}\sqrt{\Delta T - \Delta T_1}$, and near $\Delta T_2$ as $|z(t)| \propto \sqrt{\Delta T_2 - \Delta T}$.
\end{corollary}

\medskip
Before giving the proof, we formulate the result for unfolding the codimension-2 case $\Delta T_1=\Delta T_2$, where the critical eigenvalues do not transversely cross the imaginary axis.

\begin{theorem}\label{theorem3}
Assume that Hypothesis~\ref{hypoth1} holds for a fixed parameter set for which $\Delta T_1=\Delta T_2$.  Then the steady state $u=0$ of system \eqref{eq3}  possesses a locally exponentially attracting and locally invariant 2D center manifold near $u=0$ with the reduced dynamics
\begin{align}\label{eq9deg}
\frac{d z} {dt} &= \rmi (\omega+ a_0\mu_1) z + a_1\mu_1\left(a_2\mu_2 - a_3\mu_1\right) z + b z |z|^2 +  \mathcal{R}\\
\mathcal{R} &= \calO\left(\mu_2^2+ |\mu_1\mu_2^3| + |z|(|\mu|+|z|^2)^2\right),\nonumber
\end{align}
where $a_j\in\R$, $j=0,1,2,3$, are given by $a_0 = \Im(a)$,
\[
a_1 = \frac{\Re(a)\pi}{L_2(c_2-c_1)},\; a_2= \frac 1{\sqrt{\pi L_1L_2}}, \; a_3= \frac{\pi}{L_2^2},
\]
and $a$, $b$ are the constants from Theorem~\ref{theorem2}.

In particular, for $0<|\mu_2|\ll 1$, there exists a branch of stable periodic orbits, that is parametrized by $\mu_1$ and that terminates in supercritical Andronov-Hopf bifurcations at $\DT+\mu_1=\Delta T_j$, $j=1,2$.
\end{theorem}

The following Theorem shows, that the bifurcation results \ref{theorem2}-\ref{theorem3} can be generalized to instabilities caused by higher spatial harmonics.

\begin{theorem}
Assume \eqref{eq5} holds with critical wavenumber $k_2$, so that $\lambda_{(1,k_2)}^+ = \rmi\omega$. If $\DT_1(k_2^2)<\DT_2(k_2^2)$ then the statements of Theorem~\ref{theorem2} and Corollary~\ref{c:hopf} hold with $\DT_j$ replaced by $\DT_j(k_2^2)$, and $L_2$ replaced by $L_2/k_2$ for the coefficients on the center manifold. If $\DT_1(k_2^2)=\DT_2(k_2^2)$ then the statement of Theorem~\ref{theorem3} holds for the same modifications.
\end{theorem}

\begin{proof}
Under condition \eqref{eq5} the center manifold theorem applies as in the first parts of the proofs of Theorems~\ref{theorem2} and~\ref{theorem3}. This yields a stable locally invariant manifold with reduced dynamics of Hopf normal form. The only remaining question is the sign of the coefficients. 

If $k_2$ is the critical wavenumber in $x_2$-direction on the domain $[0,L_1]\times[0,L_2]$ then $1$ is this wavenumber on the domain $[0,L_1]\times[0,L_2/k_2]$ so that Hypothesis~\ref{hypoth1} holds there. Hence, on this domain and with the modifications in the claim, Theorem~\ref{theorem2}, Corollary~\ref{c:hopf} and Theorem~\ref{theorem3} hold fully.  

The theorem now follows since the bifurcating branches imbed into the original domain.
\end{proof}

\begin{remark}
Recall that there is a sequence of secondary Andronov-Hopf instabilities as noted in Remark~\ref{r:sndHopf}. Whenever these occur with a simple pair of complex conjugate eigenvalues, analogous center manifold reduction results hold for an unstable 2D manifold. The reduced vector fields are of the same form with coefficients given analogous to the above results, but to be computed at different $k_2$ and other parameters.
\end{remark}

\medskip

We start with the proof of Theorem \ref{theorem2}.
\begin{proof}[Theorem~\ref{theorem2}]
For the unfolding with $\mu_1$, we modify the definition of $R$ in \eqref{eq3} by adding the term $\mu_1 \partial_{x_2} u_1$ in the first component and denote the result by $R(u;\mu_1)$. 
For the resulting bifurcation problem, we verify the hypotheses of the center manifold theorem \cite[Theorem 3.3, p.46]{iooss}. 

As noted after \eqref{eq3}, $\L\in \mathcal{L}(Z^2,X^2)$ is sectorial so that Hypothesis 2.7 in that theorem holds, using\cite[Remark 2.18 p.\ 37]{iooss}. Hypotheses 3.1(i) and 2.4 hold on account of Theorem~\ref{theorem1}.
It remains to show Hypotheses 3.1(ii): smoothness of $R$. From \eqref{eq3} we explicitly compute
\begin{equation*}
\begin{aligned}
DR(u;\mu_1) v &= \begin{pmatrix}
\mu_1 \partial_{x_2} v_1 - A^{\perp}(v_1+v_2) \cdot \nabla u_1 - A^{\perp}(u_1 + u_2) \cdot \nabla v_1 \\
-A^{\perp}(v_1+v_2) \cdot \nabla u_2 - A^{\perp}(u_1 + u_2) \cdot \nabla v_2
\end{pmatrix},\\
D^2R(u;\mu_1)[v,w] &= - \begin{pmatrix}
A^{\perp}(v_1 + v_2) \cdot \nabla w_1 + A^{\perp}(w_1 + w_2) \cdot \nabla v_1 \\
A^{\perp}(v_1 + v_2) \cdot \nabla w_2 + A^{\perp}(w_1 + w_2) \cdot \nabla v_2
\end{pmatrix}.
\end{aligned}
\end{equation*}
Note that $A^{\perp}(v_1+v_2) \in Z^2$ and $\nabla w \in Y^2$. Since $\Hspace^2$ is a Banach algebra (see for instance \cite[Theorem (4.39]{Sobolev}), there is a constant $C_0>0$, such that
\begin{equation*}
\Vert A^{\perp}(v_1+v_2) \cdot \nabla w \Vert_Y \leq C_0 \Vert A^{\perp}(v_1+v_2) \Vert_{Y^2} \Vert \nabla w \Vert_{Y^2}.
\end{equation*}
Hence $\Vert D^2 R(u;\mu_1)[v,w] \Vert_{Y^2} \leq C \Vert v \Vert_{Z^2} \Vert w \Vert_{Z^2}$, that is, $R(u; \mu_1) \in C^2(Z^2,Y^2)$.
Moreover all the higher derivatives are identically $0$, hence $R$ is analytic.
This establishes the existence of the 2D center manifold and smoothness of $\psi$ as needed below, and for which the reduced dynamics has the normal form \eqref{eq9}. Here the critical frequency is $\omega = \pi\DT/L_2$ due to Theorem~\ref{theorem1}. 
In order to analyze the coefficients of the reduced equation, we write functions in the central subspace as
\begin{equation*}
u_0(t) = z(t) \zeta + \overline{z(t) \zeta},\mbox{ }z(t) \in \C.
\end{equation*}
Using the expressions in  \cite[p.\ 125]{iooss} (see also Scholarpedia on Andronov-Hopf bifurcation), we have
\begin{align}
a &= \langle R_{11}(\zeta)+2R_{20}(\zeta,\psi_{001}), \zeta^\ast \rangle_2\label{e:eqa}, \\
b &= \langle 2R_{20}(\zeta,\psi_{110})+ 2R_{20}(\bar\zeta,\psi_{200})+3R_{30}(\zeta,\zeta,\bar\zeta), \zeta^\ast \rangle_2.\label{e:eqb}
\end{align}
The quantities in these expressions are defined as follows: $\zeta^\ast$ is the adjoint eigenvector to $\zeta$, the operators $R_{ik}$ are given by, see \cite[p.\ 95-96]{iooss}, 
\begin{equation}
\begin{aligned}
R_{01} &:= \partial_{\mu_1} R(0;0) = 0,\\
R_{20}[v,w]&:=\frac{1}{2}D^2R(0;0)[v,w]\\ 
&= -\frac{1}{2}\begin{pmatrix}
A^{\perp}(v_1 + v_2) \cdot \nabla w_1 + A^{\perp}(w_1 + w_2) \cdot \nabla v_1 \\
A^{\perp}(v_1 + v_2) \cdot \nabla w_2 + A^{\perp}(w_1 + w_2) \cdot \nabla v_2 \\
\end{pmatrix},\\
R_{11} v &:= \partial_{\mu_1} DR(0;0)v = \begin{pmatrix}
\partial_{x_2} v_1 & \nabla^2 v_1\\
0 & \nabla^2 v_2
\end{pmatrix},\\
R_{30} & = \frac{1}{3!} D^3 R = 0, 
\end{aligned}
\end{equation}
and the functions $\psi_{ijk}$, from the expansion of $\psi$, are the unique solutions to
\begin{equation}\label{e:psis}
\begin{aligned}
-\L \psi_{001} &= R_{01},\\
(2\omega\rmi-\L)\psi_{200} &= R_{20}(\zeta,\zeta),\\
-\L\psi_{110} &= 2R_{20}(\zeta,\bar \zeta).
\end{aligned}
\end{equation}

%%%%%%%%%%
\medskip\paragraph{Computation of a} 
Since $R_{01}=0$ and $\ker(\L) = \{0\}$, $-\L \psi_{001} = R_{01}$ implies $\psi_{001} = 0$. For this result the parameter $\mu_2$ is held fixed at zero so that, using \eqref{e:eqa}, the coefficient $a$ of the reduced system \eqref{eq9} is 
\begin{equation}\label{e:eqa2}
a = \langle R_{11}(\zeta), \zeta^{\ast} \rangle_2 = \frac{2 \pi \xi^{1} \rmi}{L_2}\langle (g_{1,1},0)^T, \zeta^{\ast}\rangle_{2},
\end{equation}
where $\zeta^{\ast}$ is the adjoint eigenfunction, satisfying
\begin{equation}\label{eq11}
\L^{\ast} \zeta^{\ast} = -\rmi \omega \zeta^{\ast}\mbox{ , }\langle \zeta, \zeta^{\ast} \rangle_{2} = 1.
\end{equation}
with the adjoint operator of $\L$ given by (using integration by parts)
\begin{equation*}
\begin{aligned}
\L^{\ast}v &= \begin{pmatrix}
-\DT \partial_{x_2} \vP + \frac{1}{L_1} B v + \nu \nabla^2 \vP \\
\hspace{20mm} \frac{1}{L_1} B v + \nu \nabla^2 \vM 
\end{pmatrix}
\mbox{ , } v \in Y_2\\
Bv(x) &= \frac{4 \rmi}{L_2 \pi} \sum_{k \in \N \times \Z} \frac{k_2}{\frac{L_2}{L_1} k_1^2 + \frac{4 L_1}{L_2}k_2^2}
\langle \vP - \vM, g_k \rangle  g_k(x).
\end{aligned}
\end{equation*}
The critical adjoint eigenfunction $\zeta^{\ast}$, as any eigenfunction of $\L^{\ast}$, has the form $\zeta^{\ast}(x) = \eta g_m(x)$, where $\eta = (\eta^1, \eta^2)\in\C^2$ is an eigenvector of $M^\ast_m$ derived from \eqref{eq16}. If $m \neq (1,1)$, then $\langle \zeta, \zeta^{\ast} \rangle_{2} = 0$, therefore $m = (1,1)$, and hence $M^{\ast}_{1,1} \eta = -\rmi \omega \eta$ so that from $\langle g_{11}, g_{11}\rangle = L_1L_2/2$ and \eqref{e:eqa2} we infer
\begin{equation}\label{e:eqa3}
a = \pi \xi^{1} \overline{\eta^1} L_1 i.
\end{equation}
Due to \eqref{eq13}, there is $\xi\in\C^2$ such that
\begin{equation}\label{eq10}
M_{1,1} \xi = \rmi \omega \xi, \xi = (\xi^1, \xi^2)^T,
\end{equation}
and using \eqref{eq_const} at the bifurcation points $\Delta T= \Delta T_j$, $j=1,2$, we have
\begin{equation}\label{e:c-rel}
c_3^2 = c_1(2c_2-c_1).
\end{equation}
Together with equation \eqref{eq11} we readily check that
\begin{equation}\label{e:cmats}
\begin{aligned}
(M_{1,1}-\rmi \omega)\xi = &\begin{pmatrix}
c_1\rmi - c_2\rmi - c_3 & -c_2\rmi \\
c_2\rmi & -c_1\rmi + c_2\rmi - c_3
\end{pmatrix}
\xi = 
0\\
(M_{1,1}^\ast+\rmi\omega) \eta= &\begin{pmatrix}
-c_1\rmi + c_2\rmi - c_3 & -c_2\rmi\\
c_2\rmi & c_1\rmi - c_2\rmi - c_3
\end{pmatrix}
\eta = 
0\\
\xi \cdot & \overline{\eta} = \frac{2}{L_1 L_2}.
\end{aligned}
\end{equation}
Due to \eqref{e:c-rel}, the eigenvectors can be chosen as
\begin{equation}\label{ckernel}
\begin{aligned}
\xi = \begin{pmatrix}
c_2\rmi \\ c_1\rmi -c_2\rmi -c_3
\end{pmatrix}\mbox{, }
\overline{\eta} = \delta \begin{pmatrix}
-c_2\rmi\\
c_1\rmi - c_2\rmi - c_3
\end{pmatrix} \\
\delta = \frac{2}{L_1 L_2} \frac{1}{(c_1\rmi-c_3)(c_1\rmi-2c_2\rmi-c_3)},
\end{aligned}
\end{equation}
where $\delta \neq 0$ provides the normalization. Therefore, \eqref{e:eqa3} yields $a = \pi c_2^2 L_1 \delta \rmi$ as claimed.

%%%%%%%%%%
\medskip\paragraph{Computation of b} 
We first show $\psi_{200} = 0$; recall \eqref{e:psis}. Thanks to \eqref{eq10}, $\zeta(x) = \xi g_{1,1}(x)$ and for  $k \in \N\times \Z$ we have
\begin{equation}\label{e:A1gk}
A_1 g_k(x) = - \frac{L_2}{\pi}\frac{k_1}{\frac{L_2}{L_1}k_1^2 + 4 \frac{L_1}{L_2} k_2^2}\phi_k(x). 
\end{equation}
A direct calculation yields $R_{20}(\zeta, \zeta) = 0$. Since $\ker(2 \rmi \omega - \L) = \{0\}$ on account of Theorem~\ref{theorem1}, the equation for $\psi_{200}$ from \eqref{e:psis} implies $\psi_{200} = 0$.
Together with $R_{30}=0$ and \eqref{e:eqb}, this means
\begin{equation}\label{e:eqb2}
b = \langle 2R_{20}(\zeta,\psi_{110}), \zeta^\ast \rangle_2.
\end{equation}
Next, we compute $\psi_{110}$ using \eqref{e:psis}. From $\zeta = \xi g_{11}$ and \eqref{defA}, \eqref{e:A2gk}, \eqref{e:Agk} as well as \eqref{e:A1gk}, straightforward calculations give
\begin{equation*}
-\L \psi_{110} = 2 R_{20}(\zeta,\overline{\zeta}) = \frac{2 i}{\frac{L_2}{L_1} + 4 \frac{L_1}{L_2}}\overline{\xi} g_{2,0}.
\end{equation*}
Since the eigenvectors $(g_k)_{k \in \N \times \Z}$ of $\L$ are mutually orthogonal and $M_{2,0}$ is a multiple of the identity, we have that $\psi_{110} = \alpha \overline{\xi} g_{2,0}$, where 
\begin{equation*}
\alpha = \frac{L_1^2 \rmi}{2 \pi^2 \nu}\frac{\xi^1 + \xi^2}{\frac{L_2}{L_1} + \frac{4 L_1}{L_2}}.
\end{equation*}
It follows, after straightforward calculations, that $R_{20}(\zeta, \psi_{110}) = \beta g_{1,1} \phi_{2,0}$, where
\begin{equation*}
\beta = \alpha \Bigg( \frac{2 \rmi (\xi^1 + \xi^2)}{\frac{L_2}{L_1} + \frac{4 L_1}{L_2}}\overline{\xi} - \rmi \frac{L_1 (\overline{\xi^1} + \overline{\xi^2})}{2 L_2} \Bigg).
\end{equation*}
Substitution into \eqref{e:eqb2} yields
\begin{equation*}
b = \langle 2 R_{20}(\zeta, \psi_{110}), \zeta^{\ast} \rangle_{2} = 2 \beta \cdot \overline{\eta} \langle g_{1,1} \phi_{2,0}, g_{1,1} \rangle =
 - \frac{L_1 L_2}{2}\beta \cdot \overline{\eta}.
\end{equation*}
Finally, we use that $\overline{\xi} \cdot \overline{\eta} = 0$, see \eqref{ckernel}, and together with $\xi \cdot \overline{\eta} = \frac{2}{L_1 L_2}$ we obtain
\begin{equation*}
b = -\frac{L_1^3}{4 \pi^2 \nu}\frac{c_1^2 + c_3^2}{\frac{L_2^2}{L_1}+4 L_1},
\end{equation*}
which concludes the proof. \end{proof}

\medskip

We now turn to the proof of Corollary~\ref{c:hopf}.

\begin{proof}[Corollary~\ref{c:hopf}]
From \eqref{eq_norm_form} and \eqref{eq_const} we readily check $b<0$. 

Writing \eqref{eq_norm_form} in terms of $c_j$ and using \eqref{ckernel}, a straightforward calculation gives \begin{align}\label{e:ac}
\Re(a) = \frac{4 \pi}{L_2}\frac{c_2^2  c_3}{|(c_1 - c_3\rmi) (c_1\rmi - 2 c_2\rmi -c_3)|} (c_2 -c_1).
\end{align}
Thanks to $c_j> 0$, $j=1,2,3$, all factors in this expression are positive, except possibly the last one, and therefore the sign of $\Re(a)$ is the sign of $c_2-c_1$. Note that 
\begin{equation}\label{e:cdD}
\frac{2\pi}{L_2}(c_2-c_1) = \partial_{\Delta T} d(\Delta T, 1)/L_1^4,
\end{equation}
and that the quadratic polynomial $d(\cdot, 1)$  has negative quadratic coefficient. Therefore, $c_1 < c_2$ at $\Delta T= \Delta T_1$ and so $\Re(a) > 0$, while at $\Delta T= \Delta T_2$ we have $c_2 < c_1$, hence $\Re(a) < 0$. We readily compute that $\Im(a) = c_3^2\Re(a) /(c_3(c_2-c_1)) = \calO(\nu^2)$. 

In conclusion, there are generic supercritical Andronov-Hopf bifurcations at both endpoints of the instability region. As usual, the local invariance of the center manifold from Theorem~\ref{theorem2} implies the existence of the claimed heteroclinic orbit between the unstable steady-state and the stable limit cycle, contained in the center manifold.

\medskip
Now consider the behaviour of $a$ and $b$ for small viscosity $0<\nu \ll 1$.  
With $c_4 := \frac{c_3}{\nu}$ we get $c_2, c_4 = \calO(1)$, and 
\begin{equation}
\begin{aligned}
a &= \frac{2 \pi c_2^2}{L_2(c_1 + \nu c_4 \rmi)((c_1 - 2 c_2 )\rmi-\nu c_4)},\\
b &= - \frac{L_1^3}{4 \pi^2 \nu}\frac{c_1^2 + \nu^2 c_4^2}{\frac{L_2^2}{L_1}+4 L_1}.
\end{aligned}
\end{equation}

\paragraph{Left endpoint of the instability region: $\Delta T_1$}
Inspecting the formula for $d_\DT$ we find $c_1 = c_2 - \sqrt{c_2^2 - \nu^2 c_4^2}$, where $c_2^2 - \nu^2 c_4^2 > 0$ by \eqref{inst_condred}. Hence,
\begin{equation*}
c_1 = \frac{c_4^2}{2 c_2^2} \nu^2 + \mathcal{O}(\nu^4),
\end{equation*}
and we obtain
\begin{equation*}
a = \frac{\pi c_2}{L_2 c_4}\frac{1}{\nu} +  \mathcal{O}(\nu), \quad
b = - \frac{L_1^4 c_4^2}{4 \pi^2 ( 4 L_1^2 + L_2^2 )} \nu + \mathcal{O}(\nu^3).
\end{equation*}
Therefore the radius of the stable limit cycle $|z(t)|$ for sufficiently small $\mu_1$ is
\begin{equation*}
\frac{2 \pi}{L_1^2 c_4} \Bigg(\frac{\pi (4 L_1^2 + L_2^2) c_2}{L_2 c_4}\Bigg)^\frac{1}{2}
\frac{1}{\nu}\mu_1^\frac{1}{2} + \mathcal{O}(\nu^\frac{1}{2}).
\end{equation*}

\paragraph{Right endpoint of the instability region: $\Delta T_2$} Here $c_1 = c_2 + \sqrt{c_2^2 - \nu^2c_4^2}$, therefore
$c_1 = 2 c_2 + \mathcal{O}(\nu^2)$ and so
\begin{equation*}
a = -\frac{c_2^2 \pi}{c_4^2 L_2}\frac{1}{\nu} +  \mathcal{O}(\nu), \quad
b = \frac{L_1^4 c_2^2}{\pi^2 (4 L_1^2 + L_2^2)}\frac{1}{\nu} + \mathcal{O}(\nu) 
\end{equation*}
hence the radius of the stable limit cycle for small $-\mu_1$ is
\begin{equation*}
\frac{\pi}{L_1^2 c_4} \Bigg(\frac{4 L_1^2 + L_2^2}{L_2}\Bigg)^\frac{1}{2}(-\mu_1)^\frac{1}{2} +\mathcal{O}(\nu^\frac{1}{2}).
\end{equation*}
This concludes the proof.
\end{proof}

\medskip
We finally provide the proof of Theorem \ref{theorem3}.
\begin{proof}[Theorem~\ref{theorem3}]
In order to unfold in $\mu_2$, we cannot cite a center manifold theorem from \cite{iooss} verbatim. The reason is that $\mu_2$ modifies the second order derivative terms, but the results in \cite{iooss} are formulated only for parameter dependence of lower order terms. However, as pointed out in \cite[Remark 3.7]{iooss}, there is no problem, if the domain of $\L$ is independent of the parameter. This is the case here as long as $\nu = \nu_\crit(1) - \mu_2^2>0$, which is valid for the purpose of unfolding from $\nu=\nu_\crit$. More precisely, the proof of \cite[Theorem 3.3, p. 46]{iooss}, which considers the phase space extended by the unfolding parameter space, applies as follows for $\nu_\crit(1)>\mu_2^2$ due to the linearity in $\mu_2^2$. Set $\mu=(\mu_1,\mu_2^2)$, $\widetilde u=(u,\mu)$ and $\widetilde\L \widetilde u=(\L+\mu_1\partial_{x_2}(u_1,0)^T -\mu_2^2\nabla^2 u,0)$ as well as $\widetilde R(\widetilde u) = (R(u),0)$. (We use $\mu_2^2$ as the parameter instead of $\mu_2$ only for more pleasant reduced equations.) For the extended problem, the parameter-free center manifold theorem applies \cite[Theorem 2.9]{iooss}.

Therefore, as in the first part of the proof of Theorem~\ref{theorem1}, we obtain existence of the center manifold and the coefficient $b$ is unchanged. Let $A$ denote the real coefficient of $z$ in the vector field on the center manifold. It remains to derive the claimed $a_j$-dependent form 
\[
A=a_1\mu_1\left(a_2\mu_2 - a_3\mu_1\right) + \calO(\mu_2^2 + |\mu_1\mu_2^3|).
\]
For this we simply note that in the present case, \eqref{e:eqa2} is replaced by the more general form
\[
A = \langle R_{11}(\zeta)\mu, \zeta^{\ast} \rangle_2 = \mu_1 a - 
\mu_2^2 \langle \nabla^2 \zeta, \zeta^\ast \rangle_{2},
\]
where $\zeta = \xi g_{1,1}$. Using $\nabla^2 g_{1,1} = -\pi^2\left( \frac{1}{L_1^2} + \frac{4}{L_2^2} \right) g_{1,1}$ as well as $\langle \zeta, \zeta^\ast \rangle_{2}=1$, we obtain
\begin{equation}\label{e:A}
A = \mu_1 a + \mu_2^2 \pi^2\left( \frac{1}{L_1^2} + \frac{4}{L_2^2} \right) = \mu_1 a + \calO(\mu_2^2),
\end{equation}
with $a$ from Theorem~\ref{theorem2}, whose dependence on $\mu_2$ is considered next.
Recall that $\nu = \nu_\mathrm{crit}-\mu_2^2$, with $\mu_2=0$ giving equality in \eqref{inst_condred}. Hence,
\[
\nu_\mathrm{crit} = \frac{2}{\pi L_1L_2 \tilde c_3^2},
\]
where $\tilde c_3$ stems from writing 
\[
d(\Delta T, 1)/L_1^4 = (2\tilde c_1 - \tilde c_2 \Delta T)\Delta T - \nu^2 \tilde c_3^2,
\]
with suitably defined $\tilde c_j$, $j=1,2,3$ (note the relation to $c_j$ in \eqref{eq_const}). Then $d(\Delta T, 1)=0$ gives
\[
\Delta T_\mathrm{crit} = \tilde c_1 + \frac{\tilde c_3}{2}\mu_2\sqrt{2\nu_\mathrm{crit}-\mu_2^2}.
\]
Using \eqref{e:ac}, \eqref{e:cdD} with $\Delta T = \Delta T_\mathrm{crit} + \mu_1$ then yields
\[
a = a_1\left(\tilde c_3\mu_2\sqrt{2\nu_\mathrm{crit} - \mu_2^2} - 2 \tilde c_2 \mu_1\right).
\]
The above formula for $\nu_\mathrm{crit}$ and expansion in $\mu_2=0$ gives claimed form of $A$, when substituting the resulting $a$ into \eqref{e:A}.

The bifurcation scenario can be immediately read off the reduced vector field.
\end{proof}

%%%%%%%%%%%%%%%%%%%%%%%
\section{Travelling wave bifurcation}\label{s:tw}

As mentioned in the introduction, due to the translation symmetry in $x_2$, the Andronov-Hopf bifurcations correspond to  periodic travelling wave bifurcations. Specifically, each periodic orbit is a steady state in a comoving frame $y_2=x_2- st$ for certain $s$. While this is somewhat folklore, for completeness we give some details. The converse is clear: periodic travelling wave bifurcations imply Andronov-Hopf bifurcations. 

First note that the effect of the co-moving variable is the introduction of an advection term $s\partial_{y_2}$ on the right hand side of the first two equations in \eqref{eq1}. Therefore, the linearization $M_k$ is replaced by
\[
M_{k,s} = M_k + sC_1(k)\mathrm{Id},
\]
where $C_1(k) = 2\pi \rmi k_2/L_2$. Hence, if $\lambda_k$ is an eigenvalue of $M_k$ then $\lambda_k+sC_1(k)$ is an eigenvalue of $M_{k,s}$ and choosing critical $k_2=+1$, the frequency at bifurcation $\omega$ is replaced by $\omega + s2\pi/L_2$. The reduced equation on the center manifold then reads
\[
\dot z = \rmi(\omega + s2\pi/L_2) + \mu_1 a + b z|z|^2,
\]
where $a$ and $b$ are unmodified since the matrices made of $c_j$ in \eqref{e:cmats} \emph{do not} depend on $s$.
Hence, for $s=s_*:=-\omega L_2/2\pi$ we find steady state supercritical pitchfork bifurcations. Note the choice $k_2=-1$ reverses the sign of $\omega$, simply leading to the complex conjugate equation. 

This argument is slightly incomplete since the spectrum of the modified $\L$ possesses a double zero eigenvalue at $s=s_*$. Hence, the coefficients on the center manifold are not immediately given by the Andronov-Hopf case used above. However, the reduced vector field on the 2D center manifold of the double zero eigenvalue reduces to a scalar equation, undergoing a pitchfork bifurcation, precisely due to the translation symmetry. In polar coordinates of the Hopf normal form, this is due to detuning the trivial angular equation, co-rotation with velocity $sC_1(1)$, to stationarity. Such reductions due to continuous symmetry also hold in more abstract contexts, see, e.g., Theorem 2.18 of \cite{iooss}, where an additional reflection symmetry is assumed.

\medskip
In the context of travelling waves, let us briefly take the perspective of pattern formation, for which the infinite strip $x\in[0,L_1]\times\R$ is the natural domain here. The linear stability analysis of the laminar in this case involves the eigenvalues $\lambda_k^\pm$ from \S\ref{s:spec} with continuous and rescaled $k_2$: these are eigenmodes in the essential spectrum given by $\lambda_k^\pm$ with $k=(k_1, L_2 k_2)$, $(k_1,k_2)\in \N\times \R$. In particular, the critical modes can only be $\lambda_{(1,L_2 k_2)}^+$, $k_2\in \R$. 

We are then lead to search for pattern-forming instabilities, and indeed, the system easily allows for the analogue of Turing-Hopf instabilities from reaction-diffusion systems, which is also well known in fluid dynamics, for instance Rayleigh-Benard convection. A detailed analysis is tedious, and we only give a numerical example in Figure~\ref{f:tur}, which is derived from that in Figure~\ref{f:counter}(a). Here the critical modes at onset of the instability on the infinite strip have wavenumber near $0.75$. The periodic solutions of \S\ref{s:center}, alias, wavetrains, are a signature of the bifurcating continuum of periodic solutions. The fact that these are supercritical suggests supercritical Turing-Hopf bifurcations.

\begin{figure}[H]
\begin{center}\includegraphics[scale = 0.55]{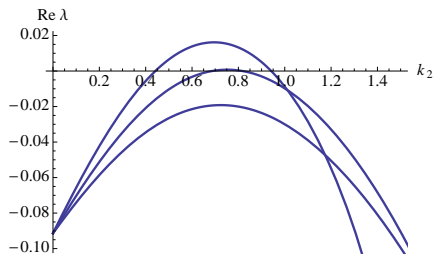}
\caption{\label{f:tur}A Turing-Hopf instability for parameters as in Figure~\ref{f:counter}(a): real parts of eigenvalues $\lambda_k^+$, $k=(1,k_2)$, as functions of $k_2$ for $\DT=0.05$ (stable), $\DT=0.08$ (near bifurcation) and $\DT=0.2$ (unstable).}
\end{center}
\end{figure}

%%%%%%%%%%%%%%%%%%%%%%%%%%%
\section{Nonlinear Instability}\label{s:nl}
In this short paragraph, we give some details on the fact that the linear instability of the laminar state $\rho_{ss}$ is indeed an instability for the nonlinear equation uniformly in $\nu$. Roughly speaking, this means that in $ (\Delta T_1, \Delta T_2)$, there are initial data which are arbitrarily close to the steady state and which get ``far'' from it exponentially quickly. We thus assume $\DT_1<\DT_2$ and take $\DT\in  (\DT_1, \DT_2)$. This means that $\sigma_+ (\L) \neq \emptyset$ (see the proof of Theorem \ref{theorem1}).

\medskip

In the parabolic formulation \eqref{eq3}, the sectoriality of $\L$ allows to apply the well-known nonlinear instability results from \cite{Henry} for spectrum in the right half plane. However, this heavily relies on $\nu>0$ and the following does not. Furthermore, the result given for the specific case here is actually stronger than the general ones in \cite{Henry}.

\medskip
As in \cite[Theorem 6.1]{daniel1}, the following instability result holds for $\nu\geq 0$:
\begin{theorem}
\label{insta}
Suppose $\Delta T \in (\Delta T_1, \Delta T_2)$. There exist constants $\delta_0,\eta_1,\eta_2>0$ such that for any $0<\delta<\delta_0$ and any $s\geq0$ there exists a solution $(\rho^\pm,E)$ to \eqref{eq1} with $\Vert \rho(0)- \rho_{ss} \Vert_{H^s} \leq \delta$ but such that:
\begin{equation*}
\Vert \rho(t_\delta)- \rho_{ss} \Vert_{\Lspace^2} \geq \eta_1
\mbox{  and  } \Vert E(t_\delta) \Vert_{\Lspace^2} \geq \eta_2,
\end{equation*}
with $t_\delta= O(\vert \log \delta \vert)$.
%In particular, $\mu$ is unstable with respect to the $L^2$ norm.
\end{theorem}

Since the proof is almost identical to that of Theorem 6.1 in \cite{daniel1}, we refer to this paper for details. 
The main idea is to apply the method of Grenier \cite{Gre}.

%%%%%%%%%%%%%%%%%%%%%%%%%%%%%%%%%%%%%%
\section{Global Nonlinear Stability}\label{s:global}

Let us now investigate the stability of the steady state $\rho_{ss}$, outside of $[\DT_1, \DT_2]$. We first state the results and then give the proofs. 

The key point is the following energy identity:
\begin{lemma}
\label{lem:id}
For any initial data $\rho_0 \in L^{\infty}$, we have the following estimate for the solution of the system \eqref{eq1}
\begin{equation}
\label{e:energy}
 \begin{aligned}
\mathcal{E}(t) := \Vert \rho-\rho_{ss} \Vert_{\Lspace^2}^2 &-\frac{2}{L_1 \Delta T}\int_{\Omega}\vert \nabla V\vert ^2 dx \\
&+ 2\nu \int_0^t\left[ \frac{-2}{L_1 \Delta T} \| \rho^+ + \rho^- - 1\|_{\Lspace^2} ^2 +  \| \nabla (\rho-\rho_{ss})\|_{\Lspace^2}^2  \right] ds \\
&\quad \quad \leq \mathcal{E}(0) =\Vert \rho_0-\rho_{ss} \Vert_{\Lspace^2}^2 - \frac{2}{L_1 \Delta T}\int_{\Omega}\vert \nabla V_{|t=0}\vert ^2 dx,
 \end{aligned}
\end{equation}
with  $\Vert \rho-\rho_{ss} \Vert_{\Lspace^2}^2=\Vert \rho^+-\rho_{ss}^+ \Vert_{\Lspace^2}^2+\Vert \rho^--\rho_{ss}^- \Vert_{\Lspace^2}^2$, 
$\Vert \nabla (\rho-\rho_{ss}) \Vert_{\Lspace^2}^2=\Vert \nabla(\rho^+-\rho_{ss}^+) \Vert_{\Lspace^2}^2+\Vert \nabla(\rho^--\rho_{ss}^-) \Vert_{\Lspace^2}^2$
and $\nabla V = \nabla (\nabla^2)^{-1} (\rho^+ + \rho^- -1)$.
\end{lemma}

\begin{remark}
We take the opportunity to point out an error in the energy of \cite[Theorem 5.1]{daniel1}: in equations (5.1) and (5.2) of this paper, there is a factor $2$ which is missing in front of $\int_{\Omega}|\nabla V|^2 dx$. 
\end{remark}

We shall use in the following  Poincar\'e type inequalities:
\begin{lemma}
\label{poinc}
With the same notations as before, we have, for any $t\geq 0$:
\begin{align}
\label{poinc1}
 \Vert \nabla V\Vert_{\Lspace^2}^2 \leq \frac{ 2L_1^2}{\pi^2 }\Vert \rho-\rho_{ss}\Vert_{\Lspace^2}^2, \\
\label{poinc2}
\Vert \rho-\rho_{ss} \Vert_{\Lspace^2}^2\leq  \frac{L_1^2}{\pi^2}  \Vert  \nabla(\rho-\rho_{ss}) \Vert_{\Lspace^2}^2.
\end{align}
\end{lemma}
As a consequence of the energy identity, we can prove $\Lspace^2$-return to equilibrium, with exponential (and explicit) speed, for negative or large enough $\Delta T$.
\begin{theorem}
\label{theo:GS}
If  $\Delta T< 0$ or ${\Delta T}>\Delta T_* := \frac{ 4 L_1}{ \pi^2}$, then the steady-state $\rho_{ss}$ is globally asymptotically stable in $\Lspace^2$, with exponential convergence given by $-2\pi^2\frac{\nu}{L_1^2}$.
\end{theorem}

\begin{remark}
Notably, the constants in all these results are independent of $L_2$, and the convergence rate is larger on thinner domains (with smaller $L_1$), but also balancing with viscosity.

Recall that by Lemma \ref{l:bound}, for fixed $\ell = L_2/L_1$ the higher instability threshold $\DT_2$ satisfies $\DT_2< \DT_*=\frac{4 L_1}{\pi^2}$, but that $\lim_{\ell \rightarrow + \infty}\Delta T_{2} = \DT_*$ if $\nu=o(\ell^{-1})$. Hence, the global threshold $\DT_*$ is also linked to linear instability.
\end{remark}

\medskip
Let us now Lemma \ref{lem:id}, Lemma \ref{poinc} and Theorem \ref{theo:GS}.

\begin{proof} [Lemma \ref{lem:id}] The proof follows from computations that are similar to those that can be found in \cite{daniel1}, for the model without viscosity (that is $\nu=0$). 
We keep the notations of Section \ref{refor}.

Taking the scalar product with $u:=(u_1,u_2)$ in the transport equations satisfied by $u_1$ and $u_2$ in \eqref{eq2}, and integrating with respect to $x$ entails:
\begin{equation}
\frac{1}{2} \frac{d}{dt} \Vert u \Vert_{\Lspace^2}^2 = \int_{\Omega}\frac{E_2}{L_1} u_1 \,  dx - \int_{\Omega}\frac{E_2}{L_1} u_2 \, dx+ \nu \int_{\Omega}u_1 \nabla^2 u_1 \, dx + \nu \int_{\Omega}u_2 \nabla^2 u_2 \, dx .
\end{equation}
Note indeed that due the periodicity with respect to $x_2$, the following contribution vanishes:
\begin{equation*}
\int_{\Omega}\partial_{x_2} u_1 u_1 dx = \int_{\Omega}\frac{1}{2}\partial_{x_2} u^2_1 dx = 0 = \int_{\Omega}\partial_{x_2} u_2 u_2 dx.
\end{equation*}
Likewise, with Green's Formula, using $\operatorname{div} E^\perp = 0$ and ${E_2}=-\partial_{x_2} V =0$ on $x_1=0,L_1$, we have (for $i=1,2$):
\begin{equation*}
\int_{\Omega}E^\perp \cdot \nabla u_i \, u_i dx = \frac{1}{2}\int_{\Omega}E^\perp \cdot \nabla (u_i)^2dx =0.
%&=&\frac{1}{2}\int_{\Omega}\operatorname{div}(E^\perp (\rho-\mu)^2)dx \\
%&=& \frac{1}{2}\left(\int_{x_1=0}E^\perp (\rho-\mu)^2.(-e_1)dx_2 + \int_{x_1=1}E^\perp (\rho-\mu)^2.e_1dx_2\right) \\
%&=& \frac{1}{2}\int_{x_1=0}-{E_2} (\rho-\mu)^2 dx_2 + \frac{1}{2}\int_{x_1=1}{E_2} (\rho-\mu)^2dx_2 \\
%&=& 0,
\end{equation*}

Recall an identity proved in \cite[Lemma 5.1]{daniel1}:
%\begin{lemma}
%\label{growth}
for any $t>0$, there holds
\begin{equation}
\label{identity}
\int_{\Omega}E_2 u_1 dx=-\int_{\Omega}E_2 u_2 dx.
\end{equation}
%\end{lemma}
For the sake of completeness, we quickly reproduce the proof.
Observe that
\begin{equation}
\begin{aligned}
\int_{\Omega}E_2 \left(u_2  - u_1 \right) dx &=  \int_{\Omega}E_2 \left(u_1 + u_2 -2 u_1 \right)dx \\
&= \int_{\Omega}E_2 \left(-\nabla^2 V-2u_1\right)dx\\
&= -2\int_{\Omega}E_2 u_1dx.
\end{aligned}
\end{equation}
Indeed, relying on the periodicity  in the $x_2$ direction and since $\partial_{x_2} V=0$ on $x_1=0,L$, we get: 
\begin{eqnarray*}
\int_{\Omega}\partial_{x_2}V  \nabla^2 V dx &=& - \int_{\Omega}\partial_{x_2} \nabla V \cdot \nabla  Vdx  + \underbrace{\int_{\Omega}\operatorname{div}(\partial_{x_2}V\nabla V )dx}_{=0}\\
&=&  -  \int_{\Omega}\partial_{x_2}\left( \frac{\vert\nabla V\vert^2}{2}\right)dx \, = \,0 .
\end{eqnarray*}
This completes the proof of \eqref{identity}. Therefore, we have:
$$
\int_{\Omega}-\frac{{E_2}}{L_1} u_1 \,  dx + \int_{\Omega}\frac{E_2}{L_1} u_2 \, dx  = - 2 \int_{\Omega}\frac{E_2}{L_1} u_1 \,  dx .
$$
Now compute, using the equations satisfied by $(u_1,u_2,V)$:
\begin{equation}
\begin{aligned}
\int_{\Omega}{E_2} u_1 \,  dx  &=  \int_{\Omega}V \partial_{x_2} u_1 dx - \underbrace{\int_{\Omega}\operatorname{div}(V u_1 e_2)dx}_{=0}\\
&= \frac{1}{\DT}\int_{\Omega}V\left(\partial_t u_1  + E^\perp \cdot \nabla u_1 - \frac{E_2}{L_1}\right)dx -  \frac{\nu}{\DT} \int_{\Omega}V \nabla^2 u_1 dx \\
&= \frac{1}{\DT}\int_{\Omega}V\left(\partial_t(u_1+u_2) + E^\perp \cdot\nabla (u_1+u_2)\right) dx \\
&\quad +   \frac{1}{\DT} \int_{\Omega}-T^- V \partial_{x_2} u_2 dx  - \frac{\nu}{\DT} \int_{\Omega}V \nabla^2 (u_1 +u_2) dx\\
&= \frac{1}{\DT}\int_{\Omega}-V\left(\partial_t\nabla^2 V - E^\perp \cdot \nabla (\nabla^2 V)\right)dx
- \frac{\nu}{\DT} \int_{\Omega}V \nabla^2 (u_1 +u_2) dx.
\end{aligned}
\end{equation}
Observe  that by Green's formula:
$$
\int_{\Omega}-V\left(\partial_t \nabla^2 V - E^\perp \cdot \nabla (\nabla^2 V)\right)dx = \frac{d}{dt} \frac{1}{2}  \left(\int_{\Omega}\vert \nabla V\vert ^2 dx \right).
$$
Finally, using again \eqref{identity}, we have
\begin{multline*}
\int_{\Omega}\frac{E_2}{L_1} u_1 \,  dx - \int_{\Omega}\frac{E_2}{L_1} u_2 \, dx  =\\
\frac{1}{L_1(T^+-T^-)}\frac{d}{dt}  \left(\int_{\Omega}\vert \nabla V\vert ^2 dx \right) dx-\frac{2\nu}{L_1(T^+-T^-)}\int_{\Omega}V \nabla^2 (u_1+ u_2) dx.
\end{multline*}
Note that using Green's formula and the Poisson equation satisfied by $V$,  we have the identities:
$$
\int_{\Omega}V \nabla^2 (u_1+ u_2) dx= \int_{\Omega} \nabla^2 V  (u_1+ u_2) dx= - \| u_1 + u_2 \|^2_{\Lspace^2}.
$$
and
$$
\nu \int_{\Omega}u_1 \nabla^2 u_1 \, dx + \nu \int_{\Omega}u_2 \nabla^2 u_2 \, dx = -\nu \int_{\Omega}|\nabla u_1|^2 \, dx  - \nu \int_{\Omega}|\nabla u_2|^2 \, dx  $$
Gathering all pieces together, we have proved that $\frac{d}{dt}\mathcal{E}(t)=0$.
%We therefore omit them and refer to \cite{daniel1}.
\end{proof}

Let us now prove the Poincar\'e inequalities of Lemma \ref{poinc}.
\begin{proof}[Lemma \ref{poinc}]
We only prove \eqref{poinc1} (\eqref{poinc2} can be treated similarly).
Using the orthogonal basis \eqref{e:gk}, we write:
$$
u_1 + u_2 =  \sum_{k_1\in \mathbb{N}_*,k_2 \in \mathbb{Z}}  a_{k_1,k_2} g_k.
$$
Recall the Poisson equation satisfied by $V$:
$$
-\nabla^2 V = u_1 + u_2.
$$
This yields:
\begin{equation*}
\begin{aligned}
V &=  \sum_{k_1 \in \mathbb{N}_*,k_2 \in \mathbb{Z}}  \frac{1}{\pi^2 \left( \frac{k_1^2}{L_1^2} + \frac{4 k_2^2}{L_2^2}   \right)}  a_{k_1,k_2}  g_k, \\ 
\nabla V &=  \sum_{k_1 \mathbb{N}_*,k_2 \in \mathbb{Z}}  \frac{1}{\pi^2 \left( \frac{k_1^2}{L_1^2} + \frac{4 k_2^2}{L_2^2}   \right)}  a_{k_1,k_2}  \pi \begin{pmatrix} \frac{k_1 }{L_1} \phi_k\\ \frac{2 \rmi k_2 }{L_2} g_k \end{pmatrix}.
\end{aligned}
\end{equation*}
Therefore,
\begin{equation}
\begin{aligned}
\int_{\Omega} |\nabla V|^2 dx &\leq   \frac{L_1^2}{\pi^2} \| u_1 + u_2\|_{\Lspace^2}^2 \leq   \frac{2L_1^2}{\pi^2} ( \| u_1\|_{\Lspace^2}^2 + \|u_2\|_{\Lspace^2}^2),
\end{aligned}
\end{equation}
which proves \eqref{poinc1}.
\end{proof}

Gathering all pieces together, we can now prove Theorem \ref{theo:GS}.
\begin{proof}[Theorem \ref{theo:GS}]
Using the energy identity \eqref{e:energy} and applying the Poincar\'e inequality \eqref{poinc1} we get:
\begin{multline*}
 \Vert \rho-\rho_{ss}\Vert_{\Lspace^2}^2 \leq \Vert \rho(0)-\rho_{ss} \Vert_{\Lspace^2}^2 + \frac{2}{L_1 \Delta T}\int_{\Omega} \vert \nabla V\vert ^2 dx \\
 -2\nu \int_0^t\left[ \frac{-4}{L_1 \Delta T} \| \rho -\rho_{ss}\|_{\Lspace^2} ^2 + \| \nabla (\rho-\rho_{ss})\|_{\Lspace^2}^2 \right] ds\\
\leq \Vert \rho(0)-\rho_{ss} \Vert_{\Lspace^2}^2 + \frac{1}{ L_1\Delta T} \frac{ 4L_1^2}{\pi^2}\Vert \rho-\rho_{ss} \Vert_{\Lspace^2}^2 \\
-2\nu \int_0^t\left[ \frac{-4}{L \Delta T} \| \rho -\rho_{ss}\|_{\Lspace^2} ^2 +  \| \nabla (\rho-\rho_{ss})\|_{\Lspace^2}^2 \right] ds.
\end{multline*}
Hence, using the Poincar\'e inequality \eqref{poinc2},
\begin{multline*}
\left(1-\frac{4 L_1}{ \pi^2\Delta T}  \right)\Vert \rho-\rho_{ss} \Vert_{\Lspace^2}^2
 \\
 \leq \Vert \rho(0)-\rho_{ss} \Vert_{\Lspace^2}^2 +2\nu  \left(\frac{4}{L_1 \Delta T} - \frac{\pi^2}{L_1^2}\right)    \int_0^t  \| \rho-\rho_{ss}\|_{\Lspace^2}^2 ds .
\end{multline*}
As a consequence, by Gronwall inequality, we obtain $\Lspace^2$-stability and $\Lspace^2$-return to equilibrium, provided that 
\[
{\Delta T}< 0, \text{  or  } {\Delta T}>\frac{ 4 L_1}{\pi^2},
\]
which in particular implies $ \frac{\pi^2}{L_1^2}- \frac{4}{L_1 \Delta T}>0$. More specifically, we have:
\begin{equation}
\Vert \rho-\rho_{ss} \Vert_{\Lspace^2}^2 \leq \Vert \rho(0)-\rho_{ss} \Vert_{\Lspace^2}^2  \exp\left(- \gamma t\right),
\end{equation}
with $\gamma :=   2 \nu  \left( \frac{\pi^2}{L_1^2 }- \frac{4}{L_1 \Delta T}\right) \left(1-\frac{4 L_1}{ \pi^2\Delta T}  \right)^{-1}\; = 2\nu\frac{\pi^2}{L_1^2} >0$.
\end{proof}

\end{document}